\documentclass[12pt,english,a4paper]{smfart}

\usepackage[utf8]{inputenc}
\usepackage[T1]{fontenc}
\usepackage[francais]{babel}
\usepackage{lmodern}
\usepackage{smfthm}
\usepackage[headings]{fullpage}
\usepackage{amssymb}

\let\cal\mathcal

\let\hat\widehat
\let\tilde\widetilde
\let\phi\varphi

\let\epsilon\varepsilon

\def\Q{{\bf Q}} 
\def\Z{{\bf Z}}
\def\C{{\bf C}}
\def\N{{\bf N}}

\def\A{{\bf A}}
\def\E{{\bf E}}

\def\O{{\cal O}}
\def\G{{\cal G}}

\def\id{{\mathrm{id}}}
\def\Aut{{\mathrm{Aut}}}
\def\Emb{{\mathrm{Emb}}}

\def\Frob{{\mathrm{Frob}}}

\def\Zp{{\Z_p}}
\def\Qp{{\Q_p}}
\def\Cp{{\C_p}}

\def\Qpbar{{\overline{\Q}_p}}

\def\At{{\tilde{\bf{A}}}}
\def\Atplus{{\tilde{\bf{A}}^+}}

\def\Btplus{{\tilde{\bf{B}}^+}}
\def\Et{{\tilde{\bf{E}}}}
\def\Etplus{{\tilde{\bf{E}}^+}}

\def\Bdr{{\bf{B}_{\mathrm{dR}}}}

\def\Gal{{\mathrm{Gal}}}

\newtheorem{question}{Question}

\author{Léo Poyeton}
\address{BICMR\\
Peking University}
\email{leo.poyeton@ens-lyon.fr}
\urladdr{perso.ens-lyon.fr/leo.poyeton/}

\date{September 2019}

\title{Formal groups and lifts of the field of norms}

\begin{document}

\begin{abstract}
Let $K$ be a finite extension of $\Qp$. The field of norms of a strictly APF extension $K_\infty/K$ is a local field of characteristic $p$ equipped with an action of $\Gal(K_\infty/K)$. When can we lift this action to characteristic zero, along with a compatible Frobenius map? In this article, we explain what we mean by lifting the field of norms, explain its relevance to the theory of $(\phi,\Gamma)$-modules, and show that under a certain assumption on the type of lift, such an extension is generated by the torsion points of a relative Lubin-Tate group and that the power series giving the lift of the action of the Galois group of $K_\infty/K$ are twists of semi-conjugates of endomorphisms of the same relative Lubin-Tate group.
\end{abstract}

\subjclass{11S15; 11S20; 11S25; 11S31; 11S82; 13F25}

\keywords{Field of norms; $(\phi,\Gamma)$-modules; $p$-adic representations; Cohen ring; non-Archimedean dynamical system; $p$-adic Hodge theory; Local class field theory; Lubin-Tate group}

\maketitle

\tableofcontents

\setlength{\baselineskip}{18pt}

\section*{Introduction}\label{intro} 
Let $K$ be a finite extension of $\Qp$, and let $K_\infty/K$ be a totally ramified Galois extension which is ``strictly arithmetically profinite'' in the sense of Wintenberger \cite{Win83}, so that we can attach to $K_\infty/K$ its field of norms $X_K(K_\infty)$. This field of norms is a field of characteristic $p$, isomorphic to $k_K(\!(\pi_K)\!)$ where $k_K$ denotes the residue field of $K$ and where $\pi_K$ is a uniformizer of $X_K(K_\infty)$, and is naturally equipped with an action of $\Gamma_K=\Gal(K_\infty/K)$. Let $E$ be a finite extension of $\Qp$ such that $k_E \supset k_K$. In this article, we consider the following slightly generalized question of Berger \cite{Ber13lifting}: when can we lift the action of $\Gamma_K$ on $k_K(\!(\pi_K)\!)$ to the $p$-adic completion of $\mathcal{O}_E[\![T]\!][1/T]$, which is a complete ring of characteristic $0$ that lifts an unramified extension of $X_K(K_\infty)$, along with a compatible $\mathcal{O}_E$-linear power of the Frobenius map $\phi_d$ ? When it is possible to do so, we say, following Berger's definition, that the action of $\Gamma_K$ is liftable. 

In the case where $k_E = k_K$ and where we have an $\mathcal{O}_E$-linear compatible Frobenius map $\phi_q$, Fontaine's construction of $(\phi,\Gamma)$-modules applies, and we get an equivalence of categories between $(\phi_q,\Gamma_K)$-modules on $\A_K$, the $p$-adic completion of $\mathcal{O}_E[\![T]\!][1/T]$, and $\mathcal{O}_E$-linear representations of $\G_K$. Such a lift is possible when $K_\infty/K$ is the cyclotomic extension, or more generally when $K_\infty/K$ is generated by the torsion points of a formal relative Lubin-Tate group, relative to a subextension  $E/F$ of $K$.

Note that, if the action of $\Gamma_K$ is liftable, we then have power series $F_g(T)$ and $P(T)$ in $\A_K$ that commute under the composition law, where $P$ is a noninvertible series and the $F_g$ are invertible. If we make the same assumption as in \cite{Ber13lifting}, that is assuming that $P(T)$ is actually a power series in $\A_K^+:=\mathcal{O}_E[\![T]\!]$, we shall call that case a lift of finite height and in that case, one can show that, up to a variable change, we have $P(T) \in T\cdot\mathcal{O}_E[\![T]\!]$ and the series $F_g(T)$ also belong to $T\cdot\mathcal{O}_E[\![T]\!]$ (see \cite[Lemm. 4.4]{Ber13lifting} for the first claim and \cite[Prop 4.2]{Ber13lifting} and the part concerning this proposition in \cite{erratumBerger} for the second claim), and Berger proved that $K_\infty/K$ is necessary abelian.

The main result of this article is the following:
\begin{theo}
\label{main theo}
If the action of $\Gamma_K$ is liftable, if $E$ is a finite Galois extension of $\Qp$ containing $K$ and if the $\mathcal{O}_E$-linear Frobenius on $\A_K$ is of finite height, then there exists a finite extension $L$ of $\Qp$, contained in $E_\infty=E \cdot K_\infty$, a subfield $F$ of $L$ and a relative Lubin-Tate group $S$, relative to the extension $F^{\mathrm{unr}} \cap L$ of $F$, such that if $L_{\infty}^S$ is the extension of $L$ generated by the torsion points of $S$, then $K_\infty \subset L_{\infty}^S$ and $L_{\infty}^S/K_\infty$ is a finite extension.
\end{theo}

In particular, this theorem puts a clear limit on which Galois extensions one can use to build a finite height one dimensional $(\phi,\Gamma_K)$-modules theory: these have to be the extensions obtained when taking the invariants under a finite group of a relative Lubin-Tate extension.  

In that finite height case, we recover the situation of non archimedean dynamical systems studied by Lubin in \cite{lubin1994nonarchimedean}, that is families of elements of $T\cdot \mathcal{O}_E[\![T]\!]$ who commute one to another for the composition law. The main ingredients of the proof of theorem \ref{main theo} are $p$-adic Hodge theory and some of the tools developed in \cite{lubin1994nonarchimedean} by Lubin to study these non archimedean dynamical systems.

In \cite[Page 341]{lubin1994nonarchimedean}, Lubin noticed that ``experimental evidence seems to suggest that for an invertible series to commute with a noninvertible series, there must be a formal group somehow in the background''. Various results have been obtained in that direction (see for instance \cite{Li96}, \cite{Li97b}, \cite{laubie2002systemes}, \cite{sarkis}, \cite{Ber17}) and our theorem \ref{main theo} shows that there is indeed a formal group that accounts for this. The last part of this article is dedicated to explicit a link between the power series by which $\Gamma_K$ acts on $\A_K$ and some endomorphims of the same relative Lubin-Tate formal group. 

As this article is dedicated to the question of lifts of finite height of the lifts of fields of norms, it also seems relevant to consider the following question:

\begin{question}
\label{question unramified}
Let $K_\infty/K$ be a strictly APF extension which is Galois, and for which the action of $\Gamma_K$ is liftable of finite height. For which finite extensions $L$ of $K$ is the action of $\Gal((L\cdot K_\infty)/L)$ liftable of finite height ?
\end{question}

In the cyclotomic case, remarks from Wach \cite[Rem. p. 393]{wach1996representations} and Herr \cite[§ 1.1.2.2]{herr1998cohomologie} show that the only extensions $L$ of $\Qp$ answering this question are the ones for which $L(\mu_{p^\infty})/\Qp(\mu_{p^\infty})$ is unramified. 

First, we prove the following, which is a partial converse of theorem \ref{main theo} and shows that a straightforward generalization of the remark of Wach and Herr does not hold for every extension for which the action is liftable:

\begin{theo}
\label{theo contredit}
Let $K$ be a finite extension of $\Qp$ and let $K_\infty/K$ be the extension generated by the torsion points of a relative Lubin-Tate group, relative to a subfield $F$ of $K$. Let $L$ be an extension of $K$, contained in $K_\infty$ such that $K_\infty/L$ is finite and Galois, of degree prime to $p$. Then there exists a finite extension $E$ of $K$ and a $\phi$-iterate extension $E_\infty/E$ such that $L=E_\infty$. 
\end{theo}

In particular, the action of $\Gal(L/E)$ and $\Gal(K_\infty/E)$ is liftable of finite height, but $K_\infty/L$ is totally ramified.

Then, we prove how one can relate the power series appearing from the $\Gamma_K$ action of $\A_K$ with endomorphisms of a formal group:

\begin{theo}
\label{theo intro isogeny}
If $K_\infty/K$ is a strictly APF Galois extension such that the action of $\Gamma_K=\Gal(K_\infty/K)$ is liftable, with power series $P(T)$ and $\{F_g(T)\} \in \A_K$, then the power series $F_g$ are twists of semi-conjugates over a finite extension of $\Qp$ to endomorphisms of a relative Lubin-Tate group, relative to a subfield $F$ of a Galois extension $E$ of $\Qp$ which contains $K$ and such that $(E \cdot K_\infty)/E$ is totally ramified. Moreover, there is an isogeny from $[\alpha]$ to $P$.
\end{theo}

In the meantime, some of the methods used to prove this theorem allow us to give an answer to question \ref{question unramified} for relative Lubin-Tate extensions that is basically the same as in the cyclotomic case:

\begin{theo}
\label{theo LT ext unramified}
Let $K_\infty/K$ be a relative Lubin-Tate extension, relative to a subfield $F$ of $K$. Then the action of $\Gamma_K$ is liftable of finite height, and the only extensions $L$ of $K$ for which there is a finite height lift of $\Gal(L \cdot K_\infty/L)$ are the ones for which $L \cdot K_\infty/K_\infty$ is unramified.
\end{theo}

\subsection*{Acknowledgements}  
Most of the results presented here are part of the author's Ph.D. thesis. Without Laurent Berger's help, encouragement and comments, none of this work would have been possible. The author would also like to thank Gabriel Dospinescu for many useful discussions.

\section{Lifting the field of norms}
Let $K$ be a finite extension of $\Qp$, and let $K_\infty$ be an infinite totally ramified Galois extension of $K$ which is ``strictly arithmetically profinite'' in the sense of \cite[Def. 1.2.1]{Win83}. This is a technical condition about the ramification of the extension $K_\infty/K$. Note that, if $\Gamma_K=\Gal(K_\infty/K)$ is a $p$-adic Lie group, it follows from the main result of \cite{sen1972ramification} that $K_\infty/K$ is strictly arithmetically profinite.  

We can apply to the extension $K_\infty/K$ the ``field of norms'' construction of Fontaine-Wintenberger \cite{Win83} which is the following~: let $\mathcal{E}_{K_\infty/K}$ be the set of finite extensions of $K$ contained in $K_\infty$, and let $X_K(K_\infty)$ be the set of sequences $(x_E)_{E \in \mathcal{E}_{K_\infty/K}}$ such that $N_{E_2/E_1}(x_{E_2}) = x_{E_1}$ whenever $E_1 \subset E_2$. By \cite[§2]{Win83}, $X_K(K_\infty)$ can be endowed with the structure of a valued field, and we have the following theorem, which is \cite[Thm. 2.1.3]{Win83}:
\begin{theo}
The field $X_K(K_\infty)$ is a local field of characteristic $p$ whose residue field is $k_K$.
\end{theo} 
In particular, if $\pi_K$ denotes a uniformizer of $X_K(K_\infty)$ we have $X_K(K_\infty) = k_K(\!(\pi_K)\!)$. By construction of $X_K(K_\infty)$, it is naturally endowed with an action of $\Gamma_K$. It is also endowed with the Frobenius map $\phi_q : x \mapsto x^q$ which commutes with the action of $\Gamma_K$.

Now let $E$ be a finite extension of $\Qp$ with residue field $k_E=k_K$, let $\varpi_E$ be a uniformizer of $E$ and let $\A_K$ denote the $\varpi_E$-adic completion of $\mathcal{O}_E[\![T]\!][1/T]$. The notation $\A_K$ is used for compatibility with the action of $\G_K$, but be careful that it is actually dependent on $E$, even though $E$ does not appear in this notation. This ring $\A_K$ is a $\varpi_E$-Cohen ring for $X_K(K_\infty)=k_K(\!(\pi_K)\!)$. Following the definition of \cite{Ber13lifting}, we will say that the action of $\Gamma_K$ is liftable if there exists such a field $E$ and power series $\{F_g(T)\}_{g \in \Gamma_K}$ and $P(T)$ in $\A_K$ such that:
\begin{enumerate}
\item $\overline{F}_g(\pi_K) = g(\pi_K)$ and $\overline{P}(\pi_K) = \pi_K^q$ ;
\item $F_g \circ P = P \circ F_g$ and $F_g \circ F_h = F_{hg}$ for all $g,h \in \Gamma_K$.
\end{enumerate}

The main question we are interested in is trying to understand for which extensions $K_\infty/K$ the action of $\Gamma_K$ is liftable. One reason for asking this is that in that case we get a $(\phi,\Gamma)$-module theory a la Fontaine to study $\mathcal{O}_E$-representations of $\G_K$, replacing the cyclotomic extension in the theory of Fontaine by the extension $K_\infty/K$. In particular, if the action of $\Gamma_K$ is liftable, then there is an equivalence of categories between étale $(\phi_q,\Gamma_K)$-modules on $\A_K$ and $\mathcal{O}_E$-linear representations of $\G_K$ (see \cite[Thm. 2.1]{Ber13lifting}).

If $L$ is a finite extension of $K$, let $L_\infty:=L\cdot K_\infty$ and $\Gamma_L=\Gal(L_\infty/K)$. By \cite[§2.3]{Win83}, every finite separable extension of $X_K(K_\infty)$ is of the form $X_K(L_\infty)$ for some $L$ finite over $K$. To such an extension $X_K(L_\infty)/X_K(K_\infty)$, there exists by Hensel lemma a unique étale extension of $\varpi_E$-rings $\A_L/\A_K$ such that $\A_L$ is a Cohen ring for $X_K(L_\infty)$.

We have the following results, which are theorems 1.3 and 1.4 of \cite{Ber13lifting}:

\begin{theo}
\label{liftable stable par ext}
If the action of $\Gamma_K$ on $X_K(K_\infty)$ is liftable and if $L/K$ is a finite extension, then the action of $\Gamma_L$ on $X_K(L_\infty)$ is liftable.
\end{theo}

\begin{theo}
Let $F_\infty \subset K_\infty$ be a Galois subextension of $K_\infty/K$ such that $K_\infty/F_\infty$ is finite, and let $\Gamma_F:=\Gal(F_\infty/K)$. If the action of $\Gamma_K$ on $X_K(K_\infty)$ is liftable then the action of $\Gamma_F$ on $X_K(F_\infty)$ is liftable.
\end{theo}

In \cite[Thm. 4.1]{Ber13lifting}, Berger showed the following:
\begin{theo}
\label{Berger poids pos ext ab}
If the action of $\Gamma_K$ is liftable with $P(T) \in \mathcal{O}_E[\![T]\!]$, then there is an injective character $\eta: \Gamma_K \to \mathcal{O}_E^{\times}$ whose conjugates by $\Emb(K,\Qpbar)$ are all de Rham with weights in $\Z_{\geq 0}$.
\end{theo}
In particular, if the action of $\Gamma_K$ is liftable with $P(T) \in \mathcal{O}_E[\![T]\!]$, then $K_\infty/K$ is abelian. We also know some cases where the action of $\Gamma_K$ is liftable, namely in the cyclotomic case, which is the situation studied by Fontaine in \cite{Fon90}, and more generally in the Lubin-Tate case (\cite{Win83} for the remark that it is liftable in characteristic zero and \cite{KR09} for the construction of the related $(\phi,\Gamma)$-modules).

We can actually generalize the notion of ``liftable action'': let $E$ be a finite extension of $\Qp$ with residue field $k_E \supset k_K$ and let $d$ be a power of $q$. We say that the action of $\Gamma_K$ is liftable if there exists power series $\{F_g(T)\}_{g \in \Gamma_K}$ and $P(T)$ in $\A_K$ such that:
\begin{enumerate}
\item $\overline{F}_g(T) \in k_K(\!(T)\!)$ and $\overline{P}(T) \in k_K(\!(T)\!)$,
\item $\overline{F}_g(\pi_K) = g(\pi_K)$ and $\overline{P}(\pi_K) = \pi_K^d$,
\item $F_g \circ P = P \circ F_g$ and $F_g \circ F_h = F_{hg}$ for all $g,h \in \Gamma_K$.
\end{enumerate}

In this more general setting, there is no reason anymore to get a theory of $(\phi,\Gamma)$-modules \textit{a priori}.

\section{Embeddings into rings of periods}
In this section, we explain how to view the different rings we talked about in the previous section as subrings of some of Fontaine's rings of periods, and we give some key results about those rings and embeddings that will be used later on. We also recall some results of Cais and Davis about their ``canonical Cohen rings for norm fields'' \cite{cais2015canonical} that we shall need later on.

Let $K/\Qp$ be a finite extension and let $K_\infty/K$ be a strictly APF extension. It will be convenient later on to see $X_K(K_\infty)$ as a subring of Fontaine's ring $\Et$ which we define as follows: since $K_\infty/K$ is strictly APF, there exists by \cite[4.2.2.1]{Win83} a constant $c = c(K_{\infty}/K) > 0$ such that for all $F \subset F'$ finite subextensions of $K_\infty/K$, and for all $x \in \mathcal{O}_{F'}$, we have 
$$v_K(\frac{N_{F'/F}(x)}{x^{[F':F]}}-1) \geq c.$$
We can always assume that $c \leq v_K(p)/(p-1)$. If $F$ is a subfield of $\Cp$, let $\mathfrak{a}_F^c$ be the set of elements $x$ of $F$ such that $v_K(x) \geq c$, and let $\Etplus := \varprojlim_{x \mapsto x^d}\mathcal{O}_{\Cp}/\mathfrak{a}_{\Cp}^c$. By §2.1 and §4.2 of \cite{Win83}, there is a canonical $\G_K$-equivariant embedding $\iota_K : A_K(K_\infty) \hookrightarrow \Etplus$, where $A_K(K_\infty)$ is the ring of integers of $X_K(K_\infty)$. We can extend this embedding into a $\G_K$-equivariant embedding $X_K(K_\infty) \hookrightarrow \Et$ where $\Et$ is the fraction field of $\Etplus$, and we note $\E_K$ its image.

It will also be convenient to have the following interpretation for $\Etplus$: 
$$\tilde{\bf{E}}^+ = \varprojlim\limits_{x \to x^q} \mathcal{O}_{\C_p} = \{(x^{(0)},x^{(1)},\dots) \in \mathcal{O}_{\C_p}^{\N}~: (x^{(n+1)})^q=x^{(n)}\}.$$
To see that both definitions coincide, we refer to \cite[Prop. 4.3.1]{BrinonConrad}.

Note that, even though $\E_K$ depends on $K_\infty$ rather than on $K$, it is still sensitive to $K$:
\begin{prop}
\label{extension corps des normes bas depend de K}
Let $K'$ be a finite extension of $K$ contained in $K_\infty$. Let $K_1$ (resp. $K_1'$) be the maximal tamely ramified extension of $K_\infty/K$ (resp. $K_\infty/K'$). Then as subfields of $\Et$, $\E_{K'}$ is a purely inseparable extension of $\E_K$ of degree $[K_1':K_1]$. In particular, $\E_{K_1}=\E_K$.
\end{prop}
\begin{proof}
See \cite[Prop. 4.14]{cais2015canonical}.
\end{proof}

Assume now that $K_\infty/K$ is Galois and that the action of $\Gamma_K$ is liftable to $\A_K$, a $\varpi_E$-Cohen ring of $\E_K$. Let $\At = \mathcal{O}_E \otimes_{\mathcal{O}_{E_0}}W(\Et)$ and endow it with the $\mathcal{O}_E$-linear Frobenius map $\phi_q$ and the $\mathcal{O}_E$-linear action of $\G_K$ coming from those on $\Et$ (this is well defined since $K \supset E_0$). Let $\At_K = \At^{\Gal(\Qpbar/K_\infty)}$. 

\begin{prop}
\label{uembedding}
There is a $\G_K$-equivariant embedding $\A_K \hookrightarrow \At_K$ and compatible with $\phi_q$ that lifts the embedding $X_K(K_\infty) \hookrightarrow \Et_K=\Et^{\Gal(\Qpbar/K_\infty)}$.
\end{prop}
\begin{proof}
See \cite[A.1.3]{Fon90} or \cite[§3]{Ber13lifting}.
\end{proof}

To simplify the notations, we will still denote $\A_K$ for its image by this embedding. As stated in the previous section, we have the following lemma which is a consequence of Hensel's lemma:

\begin{lemm}
\label{lemm existunique extension étale cohen}
Let $L$ be a finite extension of $K$. Then there exists a unique étale extension of $\varpi_E$-rings $\A_L/\A_K$ in $\At$ that lifts $\E_L/\E_K$.
\end{lemm}
\begin{proof}
See for example the discussion before \cite[Thm. 1.3]{Ber13lifting}.
\end{proof}

Let $u \in \At_K$ be the image of $T$ by the embedding given by proposition \ref{uembedding}, so that $\phi_q(u) = P(u)$ and $g(u) = F_g(u)$ for $g \in \Gamma_K$. Assume that the lift is of finite height, that is $P(T) \in \mathcal{O}_E[\![T]\!]$. Then we have the following:

\begin{lemm}
\label{lemm hauteur finie Atplus}
If $P(T) \in \mathcal{O}_E[\![T]\!]$ then $u \in \Atplus$.
\end{lemm}
\begin{proof}
See \cite[Lemm. 3.1]{Ber13lifting}.
\end{proof}

\begin{defi}
If $\A_K \subset \At$ is a $\varpi_E$-Cohen ring of $\E_K$, we define $\A_K^+:=\A_K \cap \Atplus$, where $\Atplus = \mathcal{O}_E \otimes_{\mathcal{O}_{E_0}}W(\Etplus)$.
\end{defi}

Even though the assumption that the lift is of finite height might seem reasonable, it is a really strong assumption, as it is not stable under base change. Consider for example the cyclotomic case over $\Qp$, so that we can take $P(T) = (1+T)^p - 1$ and $F_g(T) = (1+T)^{\chi(g)}-1$, where $\chi$ is the cyclotomic character, and $u = [\epsilon]-1 \in \Atplus$. We let $\A_{\Qp}$ be the $p$-adic completion of $\Z_p[\![u]\!][1/u]$ and $\A_{\Qp}^+ = \Zp[\![u]\!]$. Consider $K$ a finite extension of $\Qp$, and let $K_\infty = K(\mu_{p^\infty})$, the cyclotomic extension of $K$. Let $\Gamma_K=\Gal(K_\infty/K)$, and let $\A_K$ be the unique étale extension of $\A_\Qp$ corresponding to the extension $K_\infty/\Qp(\mu_{p^\infty})$ by the field of norms theory and lemma \ref{lemm existunique extension étale cohen}. Lemma \ref{lemm hauteur finie Atplus} shows that, if there exists a finite height lift of $K_\infty/K$, then $\A_K^+$ generates $\A_K$, in the sense that there exists $v \in \A_K^+$ such that $\A_K$ is the $p$-adic completion of $\mathcal{O}_E[\![v]\!][1/v]$. However, Wach's remark in \cite[Rem. p. 393]{wach1996representations} and Herr's remark in \cite[§ 1.1.2.2]{herr1998cohomologie} show that the only finite extensions $K$ of $\Qp$ for which the action of $\Gal(K_\infty/K)$ is liftable are the ones for which $K_\infty/\Qp(\mu_{p^\infty})$ is unramified.

One can wonder if Wach's and Herr's remarks still hold in any case of a finite height lift:

\begin{enonce}{Question}
Let $K_\infty/K$ be a strictly APF extension which is Galois and such that the action of $\Gamma_K = \Gal(K_\infty/K)$ is liftable of finite height on $\A_K$. What are the finite extensions $L$ of $K$ such that $\A_L^+$ generates $\A_L$ ? Are they only the extensions $L$ of $K$ for which $L_\infty/K_\infty$ is unramified ?
\end{enonce}

We finish this section with some results about the canonical Cohen rings for norm fields of Cais and Davis \cite{cais2015canonical}. They defined a ring $\A_{K_\infty/K}^+$, canonically attached to a strictly APF extension $K_\infty/K$:

\begin{defi}
\label{defi canonical CaisDavis}
Let $K$ be a finite extension of $\Qp$, let $\pi$ be a uniformizer of $\mathcal{O}_K$, let $K_\infty/K$ be a strictly APF extension, and let $\{K_m\}_{m \geq 0}$ be the tower of elementary extensions of $K_\infty/K$ as in \cite[§1.3]{Win83}. Define
$$\A_{K_\infty/K}^+:=\left\{(\underline{x}_{q^i})_i \in \varprojlim W_\pi(\mathcal{O}_{\widehat{K_\infty}}) : \underline{x}_{q^j} \in W_\pi(\mathcal{O}_{K_m}) \textrm{ whenever } q^j|[K_m:K_1] \right\},$$
viewed as a subring of $\mathcal{O}_K \otimes_{\mathcal{O}_{K_0}}W(\Etplus)$ \textit{via} \cite[Prop. 5.1]{cais2015canonical}, and where the $W_\pi$ are some generalized Witt vectors defined in \cite[§3]{cais2015canonical}.
\end{defi}

\begin{prop}
\label{prop embedding canonical Cohen ring}
Let $L_\infty$ be a finite extension of $K_\infty$ with wild ramification degree $a$, and let $b \in \N$ be such that $a \leq q^b$. Then the map 
$$\A_{K_\infty/K}^+ \rightarrow \A_{L_\infty/K}^+$$
given by
$$(\underline{x}_{q^i})_i \mapsto (\underline{x}'_{q^j})_j, \quad \underline{x}'_{q^{j+b}}:=\underline{x}_{q^j}$$
is a well defined embedding.
\end{prop}
\begin{proof}
The proof is the same as the first part of the proof of \cite[Thm. 6.9]{cais2015canonical}.
\end{proof}

\begin{rema}
Note that this ring $\A_{K_\infty/K}^+$ is \textit{a priori} not related to the ring $\A_K$.
\end{rema}

\section{Relative Lubin-Tate groups}
\label{LTrelatif}
In this section, we quickly recall de Shalit's construction (see \cite{de1985relative}) of a family of formal groups that generalize Lubin-Tate groups and the results we will need about them.

Let $F$ be a finite extension of $\Qp$, with ring of integers $\mathcal{O}_F$ and residue field $k_F$ of cardinal $q$. Let $h \geq 1$ and let $E$ be the unramified extension of $F$ of degree $h$. Let $\phi_q:E \to E$ be the $F$-linear Frobenius map that lifts $[x \mapsto x^q]$. If $f(T) = \sum_{i \geq 0}f_iT^i \in E[\![T]\!]$, let $f^{\phi_q}(T) = \sum_{i \geq 0}\phi_q(f_i)T^i$.

For $\alpha \in \mathcal{O}_F$ such that $v_F(\alpha) = h$, let $\mathcal{F}_{\alpha}^r$ be the set of power series $f(T) \in \mathcal{O}_E[\![T]\!]$ such that $f(T) = \pi T+ O(T^2)$ with $N_{E/F}(\pi) = \alpha$ and such that $f(T) \equiv T^q \mod \mathfrak{m}_E[\![T]\!]$. Then $\mathcal{F}_{\alpha}^r$ is non empty since $N_{E/F}(E^{\times}) = \{x \in F^{\times}, v_F(x) \in h\cdot \Z \}$ and we have the following results:
\begin{theo}
\label{Résultats LTrelatif}
If $f(T) \in \mathcal{F}_{\alpha}^r$ then:
\begin{enumerate}
\item there is a unique formal group law $S(X,Y) \in \mathcal{O}_E[\![X,Y]\!]$ such that $S^{\phi_q} \circ f = f \circ S$, and the isomorphism class of $S$ depends only on $\alpha$,
\item for all $a \in \mathcal{O}_F$, there is a unique power series $[a](T) \in \mathcal{O}_E[\![T]\!]$ such that $[a](T) = aT + O(T^2)$ and $[a](T) \in \mathrm{End}(S)$.
\end{enumerate}
Let $x_0 = 0$ and for $m \geq 1$, let $x_m \in \Qpbar$ such that $x_m = x_{m-1}$ with $x_1 \neq 0$. Let $\Lambda_m = \{x \in \Qpbar~: [\pi^m](x)=0 \}$ and $\Lambda = \bigcup \Lambda_m$ and let $E_m = E(x_m)$ and $E_{\infty}^S = \bigcup_{m \geq 1} E_m$. We have:
\begin{enumerate}
\item $E_m = E(\Lambda_m)$ and the fields $E_m$ depend only on $\alpha$ and not on the choice of $f(T) \in \mathcal{F}_{\alpha}^r$.
\item The extension $E_m/E$ is Galois with Galois group isomorphic to $(\mathcal{O}_F/\mathfrak{m}_F^m)^{\times}$.
\item If $g \in \Gal(\Qpbar/E)$, there is a unique $\chi_{\alpha}(g) \in \mathcal{O}_F^{\times}$ such that for all $\lambda \in \Lambda$, $g(\lambda) = [\chi_{\alpha}(g)]_f(\lambda)$. 
\item $\Gal(E_{\infty}^S/E) \simeq \mathcal{O}_F^{\times}$ and that isomorphism is given by $g \mapsto \chi_{\alpha}(g)$.
\item $E_{\infty}^S \subset F^{\mathrm{ab}}$ and $E_{\infty}^S$ is the subfield of $F^{\mathrm{ab}}$ cut out by $\langle \alpha \rangle \subset F^{\times}$ by local class field theory.
\end{enumerate}
\end{theo}
\begin{proof}
See \cite{de1985relative} and \cite{iwasawa1986local}.
\end{proof}

In order not to confuse relative Lubin-Tate characters with the classical ones, we will note if needed $\chi_{\pi}^F$ the classical Lubin-Tate character associated with the uniformizer $\pi$ of $F$. If $u \in \mathcal{O}_E^\times$, we note $\mu_u^E$ the unramified character of $\G_E$ that sends $\mathrm{Frob}_E$ to $u$.

\begin{lemm}
\label{classfield}
With these notations, we have $\chi_{\alpha}(g) = \chi_{\varpi}^F(g)\cdot\mu_u^E(g)$ where $\alpha = \varpi^hu$. In particular, the action of $\Gal(\Qpbar/E)$ on the torsion points of $S$ is given by $g(x) = [\chi_\varpi^F\cdot\mu_u^E(g)](x)$.
\end{lemm}
\begin{proof}
Every part of the proof is basically in \cite[§4]{iwasawa1986local} but this result is not actually stated in \cite{iwasawa1986local}. For the convenience of the reader, we give a full proof of this result.

Let $\pi \in E$ such that $N_{E/F}(\pi) = \alpha = \varpi^hu$. We also write $\phi$ and $\phi'$ for $\Frob_F$ and $\Frob_E$ respectively. Let $f(T) = \varpi T+T^q$ and $f'(T) \in \cal{F}_\alpha$. By \cite[Prop. 4.2]{iwasawa1986local}, there is a unique formal group law $S_f(X,Y)$ (resp. $S_{f'}(X,Y)$) on $\mathcal{O}_{\hat{Q_p^{\mathrm{unr}}}}$ such that $f \circ S_f = S_f^\phi \circ f$ (resp. $f' \circ S_{f'} = S_{f'}^{\phi'} \circ f'$). We write $\Lambda$ and $\Lambda'$ for the torsion points of $S_f$ and $S_{f'}$ respectively.

By the discussion after proposition 4.4 of \cite{iwasawa1986local} and \cite[Prop. 4.5]{iwasawa1986local}, there is a unique power series $\theta \in \mathcal{O}_{\hat{Q_p^{\mathrm{unr}}}}[\![T]\!]$ such that $\theta(T) = \epsilon T +O(T^2)$, where $\epsilon$ is a unit of $\mathcal{O}_{\hat{Q_p^{\mathrm{unr}}}}$, and such that:
\begin{enumerate}
\item $f' \circ \theta = \theta^{\phi} \circ f$,
\item $S_f^\theta = S_{f'}$,
\item $[a]_f^\theta = [a]_{f'}$ for all $a \in \mathcal{O}_F$.
\end{enumerate}
The power series $\theta$ is thus such that
$$f' \circ \theta = \theta^{\phi} \circ f.$$
Since $\theta$ induces an isomorphism $\cal{F}_f \to \cal{F}_{f'}$ ($\theta$ being invertible in $\mathcal{O}_{\Q_p^{\mathrm{unr}}}$ since $\theta'(0) = \epsilon$ is a unit), we get an isomorphism $\Lambda \to \Lambda'$ of $\mathcal{O}_F$-modules, and we have
$$[a]_f^{\theta} = [a]_{f'}.$$
Moreover, we have
$$\theta^{\phi'} = \theta \circ [u]_f$$
by \cite[Lemm. 5.14]{iwasawa1986local}. Now let $x \in \Lambda$ and $g \in \Gal(\Qpbar/E)$ such that $g$ acts as $\phi'^k$ on $\Q_p^{\mathrm{unr}}$. Then 
$$g(\theta(x)) = \theta^{\phi'^k}\circ g(x)$$
i.e.
$$[\chi_{\pi}(g)]_{f'}(\theta(x)) = \theta\circ [u]_f^k\circ[\chi_{\varpi}(g)]_f(x)$$
and since this holds for every $x \in \Lambda$,
$$[\chi_{\pi}(g)]_{f'}=[\chi_{\varpi}(g)u^k]_f^{\theta} = [\chi_{\varpi}(g)u^k]_{f'}.$$
Since $g$ acts on $\Lambda'$, that is on the torsion points of $S$, \textit{via} $g(x) = [\chi_{\pi}(g)]_{f'}(x)$, the previous equality tells us that $g$ acts on the torsion points of $S$ by $[\chi_{\varpi}(g)u^k]_{f'}=[\chi_{\varpi}^F(g)\mu_u^E(g)]_{f'}$.
\end{proof}~

\section{$\phi$-iterate extensions}
\label{extphiit}
In this section, we give a notion of ``$\phi$-iterate extension'' that generalizes the Lubin-Tate extensions and is a bit more general that the definitions of \cite{Ber14iterate} and \cite{cais2015canonical}. The notion of $\phi$-iterate extension will prove useful afterwards.

\begin{defi}
\label{defphiit}
Let $u_0 = \pi$ be a uniformizer of $\mathcal{O}_K$, let $d$ be a power of $q$ and let $P(T) \in \mathcal{O}_K[\![T]\!]$ a power series such that $P(0) = 0$ and $P(T) \equiv T^d \mod \pi\mathcal{O}_K$. Let $u_{n+1} \in \mathcal{O}_{\overline{K}}$ be a root of $P(T) - u_n$ and let $K_n =K(u_n)$ and $L = \bigcup K_n$.

We call extensions of this form $\phi$-iterate.
\end{defi}

\begin{rema}
This definition is a bit more general than the one of Berger in \cite{Ber14iterate} since he asks that$P(T)$ is a polynomial of the form $T^d+a_{d-1}T^{d-1}+\cdots+a_1T$ with the $a_i$ in $\mathfrak{m}_k$ for $1 \leq i \leq d-1$.\\
Our definition \ref{defphiit} is almost the same as the one of Cais and Davis in \cite{cais2015canonical}, the only difference is that they ask $d$ to be equal to $q$ in their definition.
\end{rema}

\begin{exem}~
\begin{enumerate}
\item If $P(T) = T^q$ and $\pi$ is a uniformizer of $K$,  the corresponding $\phi$-iterate extension $K_{\infty}/K$ is a Kummer corresponding to $\pi$.
\item If $S$ is a relative Lubin-Tate group, relative to an extension $E/F$ and $\alpha \in F$ as in our §1, then $E_{\infty}^S/E_1$ is $\phi$-iterate with $\phi(T) = [\alpha](T)$.
\end{enumerate}
\end{exem}

Cais and Davis actually proved the following for their definition of $\phi$-iterate extension, and the generalization to our definition is straightforward:
\begin{theo}
\label{relevement caisdavis}
Let $K_{\infty}/K$ be a $\phi$-iterate extension. Then the ring $\A_{K_\infty/K}^+$ defined in \ref{defi canonical CaisDavis} is a sub-$\mathcal{O}_K$-algebra of $\mathcal{O}_K \otimes_{W(k_K)}W(\Etplus)$, of the form $\A_{K_\infty/K}^+ = \mathcal{O}_K[\![u]\!]$, such that $X_K(K_\infty) = k_K(\!(\overline{u})\!)$, and $\mathcal{O}_K[\![u]\!]$ is endowed with the action of a Frobenius map $\phi_q$ and  an action of $\Aut(K_\infty/K)$ that commute one to another and lift the ones on $k_K(\!(\overline{u})\!)$. 
\end{theo}
\begin{proof}
See \cite[Thm. 1.2 et Thm. 1.4]{cais2015canonical}.
\end{proof}

In their article \cite{cais2015canonical}, Cais and Davis need to make the assumption that the $\phi$-iterate extension they consider is strictly APF. The following theorem, that they proved afterwards with Lubin, shows that this assumption is actually a consequence of their (and our) definition of a $\phi$-iterate extension:
\begin{theo}
A $\phi$-iterate extension is strictly APF.
\end{theo}
\begin{proof}
This is a direct consequence of the main theorem of \cite{cais2014characterization}.
\end{proof} 

In what follows, we will only consider $\phi$-iterate extensions $K_\infty/K$ that are Galois. In particular, in that case, theorem \ref{relevement caisdavis} proves that for such extensions, the action of $\Gamma_K = \Gal(K_\infty/K)$ is liftable of finite height.

\section{Lifts of finite height}
\label{relevementHT}
We now go back to the question of the strictly APF extensions $K_\infty/K$ which are Galois with Galois group $\Gamma_K = \Gal(K_\infty/K)$ and such that the action of $\Gamma_K$ is liftable of finite height. We will show in this section that such an extension is generated, up to a finite extension, by the torsion points of a relative Lubin-Tate group.

Let us now recall the hypotheses we will work with:

Let $K_\infty/K$ be a strictly APF extension which is Galois, with Galois group $\Gamma_K = \Gal(K_\infty/K)$. We suppose that the action of $\Gamma_K$ is liftable of finite height, so that there exists a finite extension $E$ of $\Qp$ with residue field $k_E \supset k_K$, $d$ a power of $q$ and power series $\{F_g(T)\}_{g \in \Gamma_K}$ and $P(T)$ in $\A_K$ such that:
\begin{enumerate}
\item $\overline{F}_g(\pi_K) = g(\pi_K)$ and $\overline{P}(\pi_K) = \pi_K^d$,
\item $F_g \circ P = P \circ F_g$ and $F_g \circ F_h = F_{hg}$ for all $g,h \in \Gamma_K$,
\item $P(T) \in \mathcal{O}_E[\![T]\!]$.
\end{enumerate}

Recall that there exists a $\G_K$-equivariant embedding $\iota_K: X_K(K_\infty) \to \tilde{\mathbf{E}}_K$ defined by \cite[Lemm. 4.11 and Coro. 4.12]{cais2015canonical}. Let $W_E(\cdot) = \mathcal{O}_E\otimes_{\mathcal{O}_{E_0}}W(\cdot)$ be the $\varpi_E$-Witt vectors. Let $\tilde{\mathbf{A}}=W_E(\tilde{\mathbf{E}})$, endowed with the $\mathcal{O}_E$-linear Frobenius map $\varphi_d$ and with the $\mathcal{O}_E$-linear action of $\G_K$ coming from those on $\tilde{\mathbf{E}}$. Let $\tilde{\mathbf{A}}_K = \tilde{\mathbf{A}}^{H_K}$ so that $\tilde{\mathbf{A}}_K = W_E(\tilde{\mathbf{E}}_K)$. We also define $\Atplus=W_E(\Etplus)$.

Recall that by proposition \ref{uembedding} there exists a $\Gamma_K$-equivariant embedding $\A_K \rightarrow \At_K$, compatible with $\phi_q$, that lifts the embedding $\iota_K : X_K(K_\infty) \rightarrow \Et_K$. Let $u$ be the image of $T$ by the embedding $\A_K \rightarrow \At_K$ so that $\phi_d(u)=P(u)$ and $u$ lifts $\iota_K(\pi_K)$. Recall that, by lemma \ref{lemm hauteur finie Atplus}, this element $u$ belongs to $\Atplus$.

Let us recall some results of section $4$ of \cite{Ber13lifting}, which allow us to improve on the regularity of the power series $P(T)$ and $F_g(T)$ for $g \in \Gamma_K$:
\begin{prop}
\label{reg}
For all $g \in \Gamma_K$, $F_g(0) = 0$.
\end{prop}
\begin{proof}
See \cite[Prop. 4.2]{Ber13lifting} and the erratum \cite{erratumBerger}.
\end{proof}

Since we have $F_g(0)=0$, we can define a character $\eta : \Gamma_K \to \mathcal{O}_E^\times$ by $\eta(g)=F_g'(0)$.

\begin{lemm}
\label{chgtvar}
If $P(T) \in \mathcal{O}_E[\![T]\!]$, then there exists $a \in \mathfrak{m}_E$ such that, if $T'=T-a$ and $u'=u-a$, then $\phi_d(u') = Q(u')$ with $Q(T') \in T' \cdot \mathcal{O}_E[\![T']\!]$.
\end{lemm}
\begin{proof}
This is the same proof as in \cite[Lemm. 4.4]{Ber13lifting}.
\end{proof}

In the following, we consider that we have made such a change of variable and we keep the previous notations, so that $P(0)=0$ and $P(u)=\phi_d(u)$. 

\begin{lemm}
\label{Pdérivé}
We have $P'(0) \neq 0$.
\end{lemm}
\begin{proof}
This is \cite[Lemm. 4.5]{Ber13lifting}.
\end{proof}

\begin{coro}
\label{injectif}
The character $\eta~: \Gamma_K \to \mathcal{O}_E^{\times}$ is injective.
\end{coro}
\begin{proof}
This is a straightforward consequence of \cite[Prop. 1.1]{lubin1994nonarchimedean} which tells us that if $Q'(0) \in \mathfrak{m}_K \setminus \{0\}$, then a power series $f(T) \in T\cdot \mathcal{O}_E[\![T]\!]$ which commutes with $Q$ is entirely determined by $f'(0)$. This implies that $F_g$ is determined by $\eta(g)$. In particular, we have $F_g(T)=T$ if and only if $g(u)=u$ and thus if and only if $g=\id$.
\end{proof}

Since we have $P'(0) \neq 0$, we can define a logarithm attached to $P$ as in \cite{lubin1994nonarchimedean}:
\begin{prop}
\label{loglubin}
There exists a unique power series $L_P(T) \in K[\![T]\!]$, holomorphic on the open unit disk and such that:
\begin{enumerate}
\item $L_P(T) = T+O(T^2)$ ;
\item $L_P \circ P(T) = P'(0)\cdot L_P(T)$ ;
\item $L_P \circ F_g(T) = \eta(g)\cdot L_P(T)$ for $g \in \Gamma_K$.
\end{enumerate}
\end{prop}
\begin{proof}
See propositions 1.2, 2.2 and 1.3 of \cite{lubin1994nonarchimedean} which also show that
$$L_P(T) = \lim\limits_{n \to +\infty} \frac{P^{\circ n}(T)}{P'(0)^n}=T\cdot \prod\limits_{n \geq 0}\frac{S(P^{\circ n}(T))}{S(0)},$$
where $P(T) = T\cdot S(T)$.
\end{proof}

We will use this logarithm to construct crystalline periods for $\eta$ as in \cite[§4]{Ber14iterate}. First, we define $\Btplus_{\mathrm{rig}}$ as the Fréchet completion of $\Atplus[1/p]$ as follows: every element of $\Btplus:=\Atplus[1/p]$ can be written as $\sum_{k \gg -\infty}\pi^k[x_k]$ where $\pi$ is a uniformizer of $\mathcal{O}_E$ and $\{x_k\}_{k \in \Z}$ is a sequence of $\Etplus$. Let $e$ be the ramification index of $E/\Qp$, and for $r \geq 0$, define a valuation $V(\cdot,r)$ on $\Btplus$ by 
$$V(x,r) = \inf_{k \in \Z}\left(\frac{k}{e}+\frac{p-1}{pr}v_{\E}(x_k)\right) \textrm{ if } x = \sum_{k \gg -\infty}\pi^k[x_k],$$
and let $V(x) = \inf_{r \geq 0}V(x,r)$. Then $\Btplus_{\mathrm{rig}}$ is defined as the completion of $\Btplus$ for the valuation $V$. If $u \in \Atplus$ is the image of $T$ by the embedding of proposition \ref{uembedding} (and belongs to $\Atplus$ by lemma \ref{lemm hauteur finie Atplus}), $L_P(u)$ converges in $\Btplus_{\mathrm{rig}}$, and we have $g(L_P(u))= \eta(g)\cdot L_P(u)$ by the third item of \ref{loglubin}. 

Let $\Sigma$ be the set of embeddings of $E$ in $\Qpbar$. If $\tau \in \Sigma$, let $n(\tau) \in \N$ be such that $\tau = \phi^{n(\tau)}$ on $k_K$, and let $u_{\tau}=(\tau\otimes\phi^{n(\tau)})(u) \in \Atplus_{\tau(E)}$. 

For $F(T) = \sum_{i \geq 0}f_iT^i \in E[\![T]\!]$, let $F^{\tau}(T) = \sum_{i \geq 0}\tau(f_i)T^i \in \tau(E)[\![T]\!]$. We then have $g(L_P^{\tau}(u_{\tau}))=\tau(\eta(g))\cdot L_P^{\tau}(u_{\tau})$ in $\Btplus_{\tau(E),\mathrm{rig}}$, with $\Btplus_{\tau(E),\mathrm{rig}}$ the Fréchet completion of $\Atplus_{\tau(E)}[1/p]$. We then get the following result, which is theorem 4.1 of \cite{Ber13lifting} :

\begin{prop}
The character $\eta: \Gamma \to \mathcal{O}_E^{\times}$ is de Rham, with nonnegative weights. 
\end{prop}
\begin{proof}
See \cite[Thm. 4.1]{Ber13lifting}.
\end{proof}

In the case where $K=E$ and $K/\Qp$ is Galois, one can actually show the following:

\begin{prop}
\label{si K=E=Galois, car cristallin}
If $K=E$ and $K/\Qp$ is Galois, then $\eta: \Gamma \to \mathcal{O}_K^\times$ is crystalline with nonnegative weights.
\end{prop}
\begin{proof}
The proof is the same as in \cite[Prop. 5.2]{Ber14iterate}.
\end{proof}

We recall an other useful result from \cite{lubin1994nonarchimedean}:
\begin{prop}
\label{commutent}
If $f(T) \in T\cdot\mathcal{O}_E[\![T]\!]$ is such that $f'(0) \in \mathfrak{m}_E \setminus \{0\}$, let $\Lambda(f)$ be the set of the roots of all the iterates of $f$. If $g(T) \in T\cdot\mathcal{O}_E[\![T]\!]$ is such that $u'(0) \in \mathcal{O}_E^\times$ and is not a root of $1$, let $\Lambda(g)$ be the set of the fixed points of all the iterates of $g$. Then if $f$ and $g$ commute, we have $\Lambda(f)=\Lambda(g)$.
\end{prop} 
\begin{proof}
See \cite[Prop. 3.2]{lubin1994nonarchimedean}.
\end{proof}

We also get the following result (compare with \cite[Prop. 5.6]{Ber14iterate}), where $r_{\tau}$ is the weight of $\eta$ at $\tau$ for $\tau \in \Sigma$:
\begin{prop}
\label{alg}
If $\tau \in \Sigma$, the following are equivalent:
\begin{enumerate}
\item $r_{\tau} \geq 1$ ;
\item $L_P^\tau(u_{\tau}) \in \mathrm{Fil}^1\Bdr$ ;
\item $\theta(u_{\tau}) \in \Qpbar$ ;
\item $\theta(u_{\tau}) \in \Lambda(P^{\tau})$.
\end{enumerate}
\end{prop}
\begin{proof}
Since $g(L_P^{\tau}(u_{\tau}))=\tau(\eta(g))\cdot L_P^{\tau}(u_{\tau})$, the equivalence between $(1)$ and $(2)$ is straightforward. If $L_P^{\tau}(u_{\tau}) \in \mathrm{Fil}^1\Bdr$, then $L_P^{\tau}(\theta(u_{\tau}))=0$ and so $\theta(u_{\tau}) \in \Qpbar$. It is clear that $(4)$ implies $(3)$. We now prove that $(3)$ implies $(4)$. Let $x=\theta(u_{\tau})$, so that $g(x) = F_g^{\tau}(x)$. If $g$ is close enough to $1$, then $g(x) = x$ and so $x \in \Lambda(F_g^{\tau}) = \Lambda(P^{\tau})$ by proposition \ref{commutent}. The only thing left to prove is that $(4)$ implies $(2)$. If there exists $n \geq 0$ such that $(P^{\tau})^{\circ n}(\theta(u_{\tau}))=0$, then $(P^{\tau})^{\circ n}(\theta(u_{\tau})) \in \mathrm{Fil}^1\Bdr$ and so $L_P^\tau(u_{\tau}) \in \mathrm{Fil}^1\Bdr$ by proposition \ref{loglubin}.
\end{proof}

Proposition \ref{alg} is almost the same as \cite[Prop. 5.6]{Ber14iterate}. However, in the $\phi$-iterate setting, Berger had a canonical element $u$ such that $\phi_d(u)=P(u)$ and $\theta(u) \in \Qpbar$. Here, we indeed have an element $u$ such that $\phi_d(u) = P(u)$, but there is no reason to expect $\theta(u) \in \Qpbar$. This is the point of our next proposition:

\begin{prop}
\label{existvtau}
There exists $\tau \in \Sigma$ such that $\theta(u_\tau) \in \Qpbar$. Moreover, if $v_n = \theta \circ \phi_d^{-n}(u_\tau)$, then $\tau(E)_\infty = \bigcup_{n \in \N}\tau(E)(\theta \circ \phi_d^{-n}(u_\tau))$.
\end{prop}
\begin{proof}
The first thing to note is that there exists $\tau \in \Sigma$ such that $r_\tau \geq 1$. Indeed, the character $\eta$ has nonnegative weights so if such a $\tau$ did not exist, $\eta$ would have all its weights equal to $0$, so that the representation given by $\eta$ would be $\Cp$-admissible and so $\eta$ would be potentially unramified by Sen's theorem \cite[§5]{sen1973lie}. Since $K_\infty/K$ is totally ramified and infinite and $\eta$ is injective, it can not be potentially unramified.

Let $v_n = \theta \circ \phi_d^{-n}(u_\tau)$ and let $K_n = K_1(v_n)$ where $K_1$ is the maximal tamely ramified extension of $K_\infty/K$. We have $\overline{v} \in \Etplus$ and we write $\overline{v} = (\overline{v}_n)_{n \in \N}$ for the components of $\overline{v}$ in $\Etplus=\varprojlim_{x \mapsto x^d}\mathcal{O}_{\C_p}/\mathfrak{a}_{\Cp}^c$. We have $v_n = \overline{v}_n$ in $\mathcal{O}_{\Cp}/\mathfrak{a}_\Cp^c$ by \cite[1.2.3]{fontaine1994corps}. In particular, the elements $v_n$ belong to $\mathcal{O}_{K_\infty}$ mod $\mathfrak{a}_{K_\infty}^c$ and therefore belong to $\mathcal{O}_{\hat{K}_\infty}$ and are invariant under the action of $\Gal(\Qpbar/K_\infty)$. Moreover, we have $\theta(u_\tau) \in \Qpbar$ so that the $v_n$ actually belong to $K_\infty$. The $v_n$ also satisfy the relations $P^{\tau}(v_{n+1})=v_n$ and $F_g^\tau(v_n) = g(v_n)$ since $\phi_d(u_\tau) = P^{\tau}(u_\tau)$ and $g(u_\tau) = F_g^\tau(u_\tau)$. Let us now prove that $\tau(E)_{\infty} = \bigcup_{n \geq 1} \tau(E)(v_n)$. Let $L = \bigcup_{n \geq 1} K_n$ and let $g \in \Gamma_K$ be such that $g_{|L} = \id_L$. Then $g(v_n) = v_n$ for all $n \in \N$, so that $F_g^\tau(v_n) = v_n$ for all $n \geq 0$. Since $F_g^\tau(T)-T$ is a power series in $\mathcal{O}_E[\![T]\!]$, it has only finitely many zeroes in the open unit disk. Since $||v_n|| \rightarrow 1$ and for all $n \geq 0$, $||v_n|| < 1$, this implies that $F_g^\tau(T) = T$ and therefore that $g(u)=u$, i.e. $g =\id$. This implies that $\tau(E)_{\infty} = \bigcup_{n \geq 1} \tau(E)(v_n)$. 
\end{proof}

Up to replacing $E$ by a finite extension of $E$, which does not change the assumption that the action of $\Gamma_K$ is liftable on $\mathcal{O}_E[\![T]\!]$, one can assume that $E/\Qp$ is Galois (and thus in particular that $\tau(E)=E$), that $E$ contains $K$ and that the extension $E_\infty:=E \cdot K_\infty$ is totally ramified over $E$. In the following, we will assume that $E$ satisfies those assumptions. Recall that our definition of $\Atplus$ is actually dependent on $E$, so that we also have to change $\Atplus$. We can always assume that $d$ is a power of the cardinal of $k_E$ (if not, we replace $P$ by one of its iterate $P^{\circ m}$, which does not change the extension $K_\infty/K$ nor the extension $E_\infty/E$, and we renumber the $v_n$ and the extensions $K_n$ and $E_n$ to take this change into account, that is choosing $v_n = \theta \circ \phi_{d}^{-n}(u_\tau)$ after having changed $d$). We then denote, for $n \in \N$, $E_{n+1} :=E(v_n)$ (there might be some confusion since we also wrote $E_1$ for the maximal tamely ramified extension of $E$ inside $E_\infty$, but we won't use the maximal tamely ramified extension of $E$ inside $E_\infty$ in what follows). We also let $\Gamma_E=\Gal(E_\infty/E)$.

\begin{lemm}
\label{extensions En}
For all $n \in \N \cup \{+\infty\}$, the extension $E_n/E$ is Galois and we have an embedding $\Gal(E_n/E) \to \Gal(K_n/K)$ given by $g \mapsto g_{|K_n}$. Moreover, the extension $E_\infty/E$ is strictly APF.
\end{lemm}
\begin{proof}
The first part of the lemma is a well knows result of Galois theory. The second part is contained in \cite[Prop. 1.2.3]{Win83}.
\end{proof}

\begin{prop}
\label{existw}
There exists $w \in \Atplus$ such that 
$$\mathcal{O}_E[\![w]\!] = \{ x \in \Atplus, \theta \circ \phi_d^{-n}(x) \in \mathcal{O}_{E_n} \textrm{ for all } n \geq 1 \}.$$
\end{prop}
\begin{proof}
Let $R = \{ x \in \Atplus, \theta \circ \phi_d^{-n}(x) \in \mathcal{O}_{E_n} \textrm{ for all } n \geq 1 \}$. It is an $\mathcal{O}_E$-algebra, separated and complete for the $\varpi_E$-adic topology, where $\varpi_E$ is a uniformizer of $\mathcal{O}_E$. Moreover, if $x \in R$, its image in $\Etplus$ belongs to $\varprojlim\limits_{x \mapsto x^d}\mathcal{O}_{E_n}/\mathfrak{a}_{E_n}^c$. 

Proposition \ref{existvtau} shows that $u_\tau \in R$, and so in particular we know that $R/\varpi_ER$ contains $k_E[\![\overline{u_\tau}]\!]$. We also know that $\varprojlim\limits_{x \mapsto x^d}\mathcal{O}_{E_n}/\mathfrak{a}_{E_n}^c \simeq k_E[\![\overline{v}]\!]$ for some $\overline{v} \in \Etplus$ by the theory of field of norms, and we consider an element $\overline{w} \in R/\varpi_ER$ with minimal valuation within
$$\{x \in R/\varpi_ER, v_\E(x) > 0 \}.$$
This set is nonempty since it contains $\overline{u_\tau}$, and so we can chose an element of minimal valuation within it since the valuation induced by the one on $\Etplus$ on $\varprojlim\limits_{x \mapsto x^d}\mathcal{O}_{E_n}/\mathfrak{a}_{E_n}^c$ is discrete. Then $R/\varpi_ER \simeq k_E[\![\overline{w}]\!]$, and so $R = \mathcal{O}_E[\![w]\!]$ for $w \in R$ lifting $\overline{w}$, since $R$ is separated and complete for the $\varpi_E$-adic topology.
\end{proof}

\begin{lemm}
\label{ajuster w}
The ring $\mathcal{O}_E[\![w]\!]$ is stable by $\phi_d$. Moreover, there exists $a \in \mathfrak{m}_E$ such that if $w'=w-a$ then there exists $S(T) \in T\cdot \mathcal{O}_E[\![T]\!]$ such that $S(w') = \phi_d(w')$ and $S(T) \equiv T^d \mod \mathfrak{m}_E$.
\end{lemm}
\begin{proof}
The set
$$\left\{x \in \Atplus, \theta \circ \phi_d^{-n}(x) \in \mathcal{O}_{E_n} \textrm{ for all } n \geq 1 \right\}$$
is clearly stable by $\phi_d$ and is equal to $\mathcal{O}_E[\![w]\!]$ by proposition \ref{existw}, so that $\phi_d(w) \in \mathcal{O}_E[\![w]\!]$ and so there exists $Q \in \mathcal{O}_E[\![T]\!]$ such that $Q(w) = \phi_d(w)$. In particular, we have $\overline{Q}(\overline{w}) = \overline{w}^d$ and so $Q(T) \equiv T^d \mod \mathfrak{m}_E$.

Now let $R(T) = Q(T+a)$ with $a \in \mathfrak{m}_E$ and let $w'=w-a$. Then $\phi_d(w')=\phi_d(w-a)=Q(v)-a = R(w')-a$ and we let $S(T) = R(T)-a$ so that $\phi_d(w')=S(w')$. For $S(0)$ to be $0$, it suffices to find $a \in \mathfrak{m}_E$ such that $Q(a)=a$. Such an $a$ exists since we have $Q(T) \equiv T^d \mod \mathfrak{m}_E$ so that the Newton polygon of $Q(T)-T$ starts with a segment of length $1$ and of slope $-v_p(Q(0))$. 

Now, we have $S(w')=\phi_d(w')$ and so $\overline{S}(\overline{w'}) = \overline{w'}^d$, so that $S(T) \equiv T^d \mod \mathfrak{m}_E$.
\end{proof}

Lemma \ref{ajuster w} shows that one can choose $w \in \left\{x \in \Atplus, \theta \circ \phi_d^{-n}(x) \in \mathcal{O}_{E_n} \textrm{ for all } n \geq 1 \right\}$ such that $\phi_d(w) = Q(w)$ with $Q(T) \in T\cdot\mathcal{O}_E[\![T]\!]$, and we will assume in what follows that such a choice has been made.

\begin{lemm}
\label{lemm H_g(w)}
The ring $\mathcal{O}_E[\![w]\!]$ is stable under the action of $\Gamma_E$, and if $g \in \Gamma_E$, there exists a power series $H_g(T) \in \mathcal{O}_E[\![T]\!]$ such that $g(w) = H_g(w)$. 
\end{lemm}
\begin{proof}
Since for all $n \in \N$, $E_n/E$ is Galois by lemma \ref{extensions En} and since the action of $\G_E$ commutes with $\phi$ and $\theta$, the set
$$\left\{x \in \Atplus, \theta \circ \phi_d^{-n}(x) \in \mathcal{O}_{E_n} \textrm{ for all } n \geq 1 \right\}$$
is stable under the action of $\Gamma_E$, and by proposition \ref{existw}, this set is equal to $\mathcal{O}_E[\![w]\!]$. In particular, if $g \in \Gamma_E$, then $g(w) \in \mathcal{O}_E[\![w]\!]$ and so there exists $H_g(T) \in \mathcal{O}_E[\![T]\!]$ such that $H_g(w) = g(w)$. 
\end{proof}

The power series $Q(T)$ and $\{H_g(T)\}$ share the same regularity properties as the power series $P$ and $\{F_g(T)\}$:
\begin{prop}
One has $Q'(0) \neq 0$ and for all $g \in \Gamma_E$, $H_g(0) = 0$.
\end{prop}
\begin{proof}
The proof is the same as for proposition \ref{reg} and lemma \ref{Pdérivé}.
\end{proof}

We now define $\kappa(g) = H_g'(0)$, so that $g \mapsto \kappa(g)$ defines a character $\kappa: \Gamma_E \to \mathcal{O}_E^{\times}$. The properties of regularity of the power series $Q$ and $\{H_g(T)\}$ allow us to prove that $\kappa$ is injective: 

\begin{coro}
\label{kappa injectif}
The character $\kappa : \Gamma_E \to \mathcal{O}_E^{\times}$ is injective.
\end{coro}
\begin{proof}
As in the proof of corollary \ref{injectif}, this follows from proposition 1.1 of \cite{lubin1994nonarchimedean}, which shows that if $Q'(0) \in \mathfrak{m}_E \setminus \{0\}$, then a power series $f(T) \in T\cdot \mathcal{O}_E[\![T]\!]$ which commutes with $Q$ is determined by $f'(0)$. Thus $H_g$ is entirely determined by $\kappa(g)$. It remains to prove that the action of $g$ on $w$ determines $g$, but if $g \in \Gamma_E$ acts trivially on $w$, then since $u_\tau \in \mathcal{O}_E[\![w]\!]$, we get $g(u_\tau)=u_\tau$, and so $g$ acts trivially on $E_\infty$.
\end{proof}

Since $Q'(0) \neq 0$, we can also define a logarithm attached to $Q$ as in proposition \ref{loglubin}: 
\begin{prop}
There exists a unique power series $L_Q(T) \in K[\![T]\!]$, holomorphic on the open unit disk and such that:
\begin{enumerate}
\item $L_Q(T) = T+O(T^2)$ ;
\item $L_Q \circ Q(T) = Q'(0)\cdot L_Q(T)$ ;
\item $L_Q \circ H_g(T) = \kappa(g)\cdot L_Q(T)$ for $g \in \Gamma_E$.
\end{enumerate}
\end{prop}

We also define ``conjugates'' of the period $w$, which are the analogs of the periods $u_\tau$: we define, for $\tau \in \Sigma=\Gal(E/\Qp)$, $w_\tau = (\tau \otimes \phi^{n(\tau)})(w) \in \Atplus$, where $n(\tau) \in \N$ is such that $\tau =\phi^{n(\tau)}$ on $k_E$. Note that even though we fixed a particular element $\tau$ of $\Sigma$ in proposition \ref{existvtau}, we still use the notation $\tau$ for any element of $\Sigma$ to avoid too much notations.

Just as in proposition \ref{si K=E=Galois, car cristallin}, we have the following result:
\begin{prop}
The character $\kappa$ is crystalline with nonnegative weights.
\end{prop}
\begin{proof}
Again, this is the same proof as in \cite[Prop. 5.2]{Ber14iterate}.
\end{proof}

We also define for $\tau \in \Sigma, w_{\tau}^n = \theta \circ \phi_d^{-n}(w_{\tau})$ and $w_{\tau}^{n,k} = (Q^{\tau})^{\circ k}(w_{n+k}^{p^{n(\tau)}})$. These $w_{\tau}^{n,k}$ will allow us to give some approximations of the elements $w_{\tau}^n$. 

\begin{lemm}
One has $\theta\circ\phi_d^{-n}(w_{\tau}) = \lim_{k \rightarrow +\infty}w_\tau^{n,k}$.
\end{lemm}
\begin{proof}
This is the same proof as lemma 5.3 of \cite{Ber14iterate}. 
\end{proof}

\begin{lemm}
\label{aprox}
For $M > 0$, there exists $j \geq 0$ such that $v_E(w_{\tau}^n-w_{\tau}^{n,j}) \geq M$ for all $n \geq 1$.
\end{lemm}
\begin{proof}
See \cite[Lemm. 5.4]{Ber14iterate}.
\end{proof}

For $\tau \in \Sigma$, let $r_{\tau}'$ be the weight of $\kappa$ at $\tau$. We then have:
\begin{prop}
\label{algphiit}
If $\tau \in \Sigma$, the following are equivalent:
\begin{enumerate}
\item $r_{\tau}' \geq 1$ ;
\item $L_Q^{\tau}(w_{\tau}) \in \mathrm{Fil}^1\Bdr$ ;
\item $\theta(w_{\tau}) \in \Qpbar$ ;
\item $\theta(w_{\tau}) \in \Lambda(Q^{\tau})$ ;
\item $w_{\tau} \in \cup_{j \geq 0}\phi_d^{-j}(\mathcal{O}_E[\![w]\!])$.
\end{enumerate}
\end{prop}
\begin{proof}
The fact that the first four items are equivalent has already been proven in the proof of proposition \ref{alg}. It thus remains to prove the equivalence with item $(5)$. The proof is similar with the proof of \cite[Prop. 5.6]{Ber14iterate}.

 If $w_{\tau} \in \cup_{j \geq 0}\phi_d^{-j}(\mathcal{O}_E[\![w]\!])$, then since $\phi_d(w_{\tau}) = Q^{\tau}(w_{\tau})$ and since $\mathcal{O}_E[\![w]\!] = \{ x \in \Atplus, \theta \circ \phi^{-n}(x) \in \mathcal{O}_{E_n} \textrm{ for all } n \geq 1 \}$ by \ref{existw}, we get that $\theta(w_{\tau}) \in \Qpbar$. 

Now let us assume that $\theta(w_{\tau}) \in \Qpbar$. It suffices to prove that there exists $j \geq 0$ such that $w_{\tau}^n \in \mathcal{O}_{E_{n+j}}$ for all $n$ by the characterisation of $\mathcal{O}_E[\![w]\!]$ of proposition \ref{existw}. Now let $M$ be such that $M \geq 1+v_E((Q^{\tau})'(w_{\tau}^n))$ for all $n \geq 0$. By lemma \ref{aprox}, there exists $j \geq 0$ such that $v_E(w_{\tau}^n-w_{\tau}^{n,j}) \geq M$ for all $n \geq 1$. Then since $w_{\tau}^n$ is a root of $Q^{\tau}(T)-w_{\tau}^{n-1}$, $w_{\tau}^n-w_{\tau}^{n,j}$ is a root of $Q^{\tau}(w_{\tau}^{n,j}+T)-w_{\tau}^{n-1}$. If $w_{\tau}^{n-1} \in \mathcal{O}_{E_{n+j-1}}$, then $R_n(T) = Q^{\tau}(w_{\tau}^{n,j}+T)-w_{\tau}^{n-1}$ belongs to $\mathcal{O}_{E_{n+j}}[T]$, and satisfies $v_E(R_n(0)) \geq M +v_E(R'(0))$. This implies that $R_n(T)$ has a unique root of slope $v_E(R_n(0))-v_E(R_n'(0)) \geq M$ by the theory of Newton polygons. This implies that this root belongs to $E_{n+j}$. In particular, since this root is $w_{\tau}^n-w_{\tau}^{n,j}$, we get that $w_{\tau}^n \in \mathcal{O}_{E_{n+j}}$, and so this finishes the proof by induction on $n$.
\end{proof}

\begin{rema}
Proposition \ref{algphiit} is exactly why we needed to work with $\mathcal{O}_E[\![w]\!]$ instead of $\mathcal{O}_E[\![u]\!]$: we get item $5$ in proposition \ref{algphiit} (compare with proposition \ref{alg}), which allows us to write the periods corresponding to the positive weights of $\kappa$ as power series in some $\phi_d^{-n}(w)$. 
\end{rema}

If $\tau$ satisfies those conditions, we can then write $w_{\tau} = f_\tau(\phi_d^{-j_{\tau}}(w))$ for some $j_{\tau} \geq 0$ and some $f_\tau \in \mathcal{O}_E[\![T]\!]$. As in \cite[Lemm. 5.7]{Ber14iterate}, we get the following lemma:

\begin{lemm}
\label{ftau}
One has $f_\tau(0)=0$, $f_\tau'(0) \neq 0$, $Q^{\tau}\circ f_\tau(T)=f_\tau\circ Q(T)$, and $H_g^{\tau}\circ f_\tau(T) = f_\tau \circ H_g(T)$.
\end{lemm}
\begin{proof}
If $w_{\tau}=f_\tau(\phi_d^{-j}(w))$, then $Q^{\tau}(w_{\tau})=Q^{\tau}\circ f_\tau(\phi_d^{-j}(w))$, and then $\phi_d(w_{\tau}) = f_\tau \circ \phi_d^{-j}(w))$, so that $Q^{\tau}\circ f_\tau(T)=f_\tau\circ Q(T)$. For the same reason, we have $H_g^{\tau}\circ f_\tau(T)=f_\tau\circ H_g(T)$. Evaluating the relation $Q^{\tau} \circ f_\tau(T) = f_\tau \circ Q(T)$ at $T=0$ gives $Q^{\tau}(f_{\tau}(0))=f_\tau(0)$ and so $f_\tau(0)$ is a root of $Q^{\tau}(T) = T$. 

Since $Q^\tau(0)=0$ and since $Q^\tau(T) \equiv T^d \mod \mathfrak{m}_E$, we have by the theory of Newton polygons that either $f_\tau(0)$ is of valuation $0$ or $f_\tau(0)=0$. But the first case can not occur since $\theta\circ \phi_d^{-n}(w_{\tau}) = f_\tau(w_{n+j})$ and so $f_\tau(0) \in \mathfrak{m}_{E}$. Moreover, $f_\tau'(0) \neq 0$, since if we write $f(T)=f_kT^k+O(T^{k+1})$ with $f_k \neq 0$, then since $Q^{\tau}\circ f_\tau(T)=f_\tau\circ Q(T)$, we get $\tau(Q'(0))f_k=f_kQ'(0)^k$ and so $\tau(Q'(0))=Q'(0)^k$, thus $k=1$ since $v_E(Q'(0)) > 0$. 
\end{proof}

Once again, we have a similar result as the one of Berger, and the following lemma is the analog of \cite[Coro. 5.8]{Ber14iterate}.
\begin{lemm}
\label{souscorps}
Let $G = \{\tau \in \Sigma \textrm{ such that } r_{\tau}' \geq 1 \}$. Then $G$ is a subgroup of $\Sigma$ and if $F=E^G$, then $\kappa(g) \in \mathcal{O}_F^\times$. Moreover, $r_{\tau}'$ is independent of $\tau \in G$.
\end{lemm}
\begin{proof}
Proposition \ref{existw} shows that $\tau=\id$ satisfies condition $(3)$ of proposition \ref{algphiit}, so that $r_\id' \geq 1$. If $\sigma,\tau$ satisfy condition $(5)$ of proposition \ref{algphiit}, then one can write $w_\sigma = f_\sigma(\phi_d^{-j_\sigma}(w))$ and $w_\tau=f_\tau(\phi_d^{-j_\tau}(w))$, so that $w_{\sigma\tau} = f_\tau^\sigma \circ f_\sigma(\phi_d^{-j_\tau-j_\sigma}(w))$, thus $\sigma\tau$ also satisfies condition $(5)$. Since $\Sigma$ is a finite group, this shows that $G$ is a subgroup of $\Sigma$.

By lemma \ref{ftau}, $Q^{\tau}\circ f_\tau=f_\tau\circ Q(T)$, and so $(Q^{\tau})^{\circ n}\circ f_\tau=f_\tau\circ Q^{\circ n}(T)$, so that 
$$\frac{1}{Q'(0)^n}(Q^{\tau})^{\circ n}\circ f_\tau(T) = \frac{1}{Q'(0)^n}f_\tau\circ Q^{\circ n}(T),$$
which implies passing to the limit that $L_Q^{\tau}\circ f_\tau(T) = f_\tau'(0)\cdot L_Q(T)$. since we also have $H_g^{\tau}\circ f_\tau(T) = f_\tau\circ H_g(T)$, we get $g(L_Q^{\tau}\circ f_\tau(v))=\tau(\eta(g))\cdot(L_Q^{\tau}\circ f_\tau(v))$. This implies that $\tau(\kappa(g))= \kappa(g)$ since $L_Q^{\tau}\circ f_\tau(v)=f_\tau'(0)\cdot L_Q(v)$. This holds for every $\tau \in H$, and so $\kappa(g) \in \mathcal{O}_F^\times$. Since $\kappa(g) \in \mathcal{O}_F^{\times}$, its weight $r_{\tau}'$ at $\tau$ depends only on $\tau_{|F}$ and so is independent of $\tau \in H$.
\end{proof}

In order to finish the proof of theorem \ref{main theo}, we will need some local class field theory: for $\lambda$ a uniformizer of $E$, let $E_{\lambda}$ be the extension of $E$ attached to $\lambda$ by local class field theory. This extension is generated by the torsion points of a Lubin-Tate formal group defined over $E$ and attached to $\lambda$, and we write $\chi_{\lambda}^K~: \Gal(E_{\lambda}/E) \to \mathcal{O}_E^{\times}$ the corresponding Lubin-Tate character. Since $E_{\infty}/E$ is abelian and totally ramified, there exists $\lambda$ a uniformizer of $\mathcal{O}_E$ such that $E_{\infty} \subset E_{\lambda}$.

\begin{prop}
\label{expression caractere eta}
One has $\kappa = \prod_{\tau \in \Sigma}\tau(\chi_{\lambda}^E)^{r_{\tau}'}$. 
\end{prop}
\begin{proof}
This is the same proof as \cite[Prop. 6.1]{Ber14iterate}.
\end{proof}

\begin{theo}
\label{eta et sous-corps relatif}
There exists $F \subset E$ and $r \geq 1$ such that $\kappa = N_{E/F}(\chi_{\lambda}^E)^r$.
\end{theo}
\begin{proof}
Let $F$ be the field given by lemma \ref{souscorps} and let $r = r_{\tau}'$ for $\tau \in G$, which does not depend on the choice of $\tau \in G$ by the same lemma. Lemma \ref{souscorps} also shows that $G = \Gal(E/F)$, and so, combining these results with proposition \ref{expression caractere eta}, we get that $\eta=N_{E/F}(\chi_{\lambda}^E)^r$.
\end{proof}

These results and lemma \ref{classfield} allow us to prove the following:
\begin{theo}
\label{thm phigamma implique LT}
If $K$ is a finite extension of $\Qp$ and if $K_{\infty}/K$ is an infinite strictly APF Galois extension with Galois group $\Gamma_K = \Gal(K_{\infty}/K)$,  and such that there exists a finite extension $E$ of $\Qp$ such that $k_K \subset k_E$, an integer $d$ which is a power of the cardinal of $k_K$ and power series $\{F_g(T)\}_{g \in \Gamma_K}$ and $P(T)$ in $\A_K$ such that:
\begin{enumerate}
\item for all $g \in \Gamma_K$, $\overline{F}_g(T) \in k_K(\!(T)\!)$ and $\overline{P}(T) \in k_K(\!(T)\!)$;
\item $\overline{F}_g(\pi_K)=g(\pi_K)$ and $\overline{P}(\pi_K) = \pi_K^d$;
\item $F_g \circ P = P \circ F_g$ and $F_g \circ F_h = F_{hg}$ for $g,h \in \Gamma_K$;
\end{enumerate}
then there exists a finite extension $L$ of $E$, a subfield $F$ of $L$ and a relative Lubin-Tate group $S$, relative to the extension $F^{\mathrm{unr}} \cap L$ of $F$, such that if $L_{\infty}^S$ is the extension of $L$ generated by the torsion points of $S$, then $K_\infty \subset L_{\infty}^S$ and $L_\infty^S/K_\infty$ is a finite extension.
\end{theo}
\begin{proof}
Let $L$ be a finite extension of $E$ such that $L$ contains $K$, $L/\Qp$ is Galois and $L_\infty/L$ is totally ramified. Let $\lambda$ be a uniformizer of $L$ such that $L_{\infty} \subset L_{\lambda}$, let $F$ be the subfield of $L$ given by theorem \ref{eta et sous-corps relatif} and let $E'= F^{\mathrm{unr}}\cap L$ and let $\pi = N_{L/E'}(\lambda)$ and $\alpha = N_{L/F}(\lambda)$, so that $\pi$ is a uniformizer of $E'$ and $\alpha = N_{E'/F}(\pi)$. Let $S$ be the relative Lubin-Tate group attached to $\alpha$, and let $L_{\infty}^S$ be the extension of $L$ generated by the torsion points of $S$. If $g \in \Gal(\Qpbar/L_{\infty}^S)$, then $N_{L/F}(\chi_{\lambda}^L(g))=1$ by lemma \ref{classfield}, and so by theorem \ref{eta et sous-corps relatif}, we get $\kappa(g) = 1$. In particular, we have $L_{\infty} \subset L_{\infty}^S$. By theorem \ref{eta et sous-corps relatif} and some Galois theory, we get:
\begin{enumerate}
\item $L_{\infty}^S$ is the field cut out by $\{g \in \G_L~: N_{L/F}(\chi_{\lambda}^L(g))=1 \}$ ;
\item $L_{\infty}$ is the field cut out by $\{g \in \G_L~: N_{L/F}(\chi_{\lambda}^L(g))^r=1\}$.
\end{enumerate}
so that $L_{\infty}^S/L_{\infty}$ is a finite Galois extension whose Galois group is isomorphic to $\{x \in \mathcal{O}_F^{\times},x^r=1\}$. The conclusion comes from the fact that $L_\infty/K_\infty$ is finite.
\end{proof}~

\section{From Lubin-Tate extensions to $\phi$-iterate extensions}
\label{LTphiit}
In this section, we show a partial converse of theorem \ref{thm phigamma implique LT}. Let $K$ be a finite extension of $\Qp$ and let $K_\infty/K$ be the extension generated by the torsion points of a relative Lubin-Tate group, relative to a subfield $F$ of $K$. We will prove that, if $L$ is an extension of $K$, contained in $K_\infty$ and such that $K_\infty/L$ is finite and Galois, of degree prime to $p$, then $L/K$ is $\phi$-iterate up to a finite level. In particular, using theorem \ref{relevement caisdavis}, the action of $L/K$ will be liftable of finite height. To do so, we use some results from \cite{laubie2002systemes}. These results will be used in the next section.

Note however that a full converse of theorem \ref{thm phigamma implique LT} does not hold. A full converse result of theorem \ref{thm phigamma implique LT} would be that when $K/\Qp$ is a finite extension and $K_\infty/K$ is the extension generated over $K$ by the torsion points of a relative Lubin-Tate group $S$, relative to an extension $E/F$ with $E \subset K$, then any extension $L$ of $K$ such that $K_\infty/L$ is finite and Galois would be $\phi$-iterate. In particular, taking $L=K_\infty$, the action of $\Gamma_K=\Gal(K_\infty/K)$ would be liftable of finite height thanks to theorem \ref{relevement caisdavis}. But in the case where $K_\infty$ is the cyclotomic extension of $K$ and $K$ is such that $K_\infty/\Qp(\zeta_{p^\infty})$ is totally ramified, \cite[Rem. p. 393]{wach1996representations} and \cite[§ 1.1.2.2]{herr1998cohomologie} show that such a lift is not possible.

In what follows, $F$ is a finite extension of $\Qp$ and $K$ is a finite unramified extension of $F$. Let $q$ be the cardinal of $k_K$. Let $S(X,Y)$ be a formal relative Lubin-Tate group, relative to the extension $K/F$, attached to $\alpha \in \mathcal{O}_F$ and $K_\infty$ is the extension of $K$ generated by the torsion points of $S$. We note $\Gamma_K=\Gal(K_\infty/K)$, $\chi_\alpha : \Gamma_K \to \mathcal{O}_F^\times$ the relative Lubin-Tate character attached to $\alpha$. Let $L$ be an extension of $K$, contained in $K_\infty$ such that $K_\infty/L$ is finite and Galois, and let $W=\Gal(K_\infty/L)$. Let $d$ be the cardinal of $W = \Gal(K_\infty/L)$ so that $\chi_\alpha(W)$ is the set of $d$-th roots of unity. To keep the notations simple, we will stille write $W$ for its image by $\chi_\alpha$. Let $M=K(\!(T)\!)$.

For $a \in K$, let $[a] = [a](T) \in M$ be the unique endomorphism of $S$ over $K$ whose derivative at $0$ is $a$. Let $P=[\alpha]$.

The group $W$ acts on $M$ by (right) composition: $(w,Q) \mapsto w \cdot Q = Q \circ [w]$.

By Artin's theorem, $M/M^W$ is Galois with Galois group isomorphic to $W$, and we define $R:=N_{M/M^W}(T) = \prod_{w \in W}[w](T)$.
 
\begin{lemm}
\label{thm Samuel}
We have $M^W = K(\!(R)\!)$.
\end{lemm}
\begin{proof}
This is lemma 2.1 of \cite{laubie2002systemes} and it is a straightforward consequence of a more general theorem of Samuel \cite{samuel1966groupes}, which says that if $\mathcal{O}$ is a local noetherian complete domain, and if $G$ is a finite group of $\mathcal{O}$-automorphisms of $\mathcal{O}[\![T]\!]$, then the ring $\mathcal{O}[\![T]\!]^G$ of $G$-invariants is the subring $\mathcal{O}[\![f(T)]\!]$, where $f = \prod_{s \in G}s\cdot T$.
\end{proof}

\begin{prop}
\label{def séries gammapi}
For all $a$ in $K$, there exists a unique power series $\Gamma_a$ such that
$$\Gamma_a \circ R = R \circ [a] = N_{M/M^W}([a])$$
We then have the relation $\Gamma_{a} \circ \Gamma_b = \Gamma_{\mathrm{ab}}$ and $\Gamma_a$ is for $a \in \mathcal{O}_K$ a power series whose derivative at $0$ is $a^d$.
\end{prop}
\begin{proof}
The fact that, for all $a \in K$, there exists a unique power series $\Gamma_a$ such that 
$$\Gamma_a \circ R = R \circ [a] = N_{M/M^W}([a])$$
is a direct consequence of lemma \ref{thm Samuel}. The relation $\Gamma_{a} \circ \Gamma_b = \Gamma_{\mathrm{ab}}$ follows from the fact that $[a] \circ [b] = [ab]$.
\end{proof}

The power series $\Gamma_k$ for $k \in \mathcal{O}_K^{\times}$ are entirely determined by their derivative at $0$:
\begin{lemm}
\label{Gamma dérivée en 0}
Let $k,\ell \in \mathcal{O}_K^{\times}$. We have $\Gamma_k = \Gamma_{\ell}$ if and only if $\Gamma_k'(0) = \Gamma_{\ell}'(0)$.
\end{lemm}
\begin{proof}
Let $L_P$ be Lubin's logarithm attached to $P$ as in proposition \ref{loglubin}. Then 
$$L_P(\Gamma_k \circ R) = L_P(R \circ [k]) = L_P\left(\prod_{w \in W}[wk]\right) = (-1)^{d-1}k^d\cdot L_P$$
by item $2$ of proposition \ref{loglubin}. In particular, $\Gamma_k = \Gamma_{\ell}$ if and only if $\ell^d=k^d$, which are the derivatives at $0$ of $\Gamma_\ell$ and $\Gamma_k$ by proposition \ref{def séries gammapi}.
\end{proof}

We will now consider the power series $\Gamma_{\alpha}$ and use it to construct a $\phi$-iterate extension. To do so, we will need the following lemma:

\begin{lemm}
\label{Gammapi phiit}
The power series $\Gamma_{\alpha}$ satisfies the following:
\begin{enumerate}
\item $\Gamma_{\alpha}(0) = 0$ and $\Gamma_{\alpha}'(0) = \alpha^d$ ;
\item $\Gamma_{\alpha}(T) \equiv T^q \mod \mathfrak{m}_K$.
\end{enumerate}
\end{lemm}
\begin{proof}
The first item has already been proven in proposition \ref{def séries gammapi}. In order to prove the second item, we use the relation $\Gamma_{\alpha} \circ R = R \circ [\alpha](T)$, which gives us by reducing mod $\pi$ : $\Gamma_{\alpha}(R) = R(T^q) \mod \pi$ and so $\Gamma_{\alpha}(R) = R(T)^q \mod \pi$. This implies that $\Gamma_{\alpha} = T^q \mod \pi$.
\end{proof}

\begin{theo}
\label{engendrée LT est phiit}
Let $K$ be a finite extension of $\Qp$ and let $K_\infty/K$ be the extension generated by the torsion points of a relative Lubin-Tate group, relative to a subfield $F$ of $K$. Let $L$ be an extension of $K$, contained in $K_\infty$ such that $K_\infty/L$ is finite and Galois, of degree prime to $p$. Then there exists a finite extension $E$ of $K$ and a $\phi$-iterate extension $E_\infty/E$ such that $L=E_\infty$. Moreover, we can write $E_{\infty} = \bigcup E(u_n)$ and there exists a family $\{F_g(T)\}_{g \in \Gal(E_{\infty}/E)}$ of power series such that $g \in \Gal(E_{\infty}/E)$ acts on the $u_n$ by $g(u_n) = F_g(u_n)$ with $F_g(0)=0$, and $g \mapsto F_g'(0)$ defines a character $\eta : \Gal(E_{\infty}/E) \to \mathcal{O}_F^{\times}$ satisfying $\eta(g) = \chi_{\alpha}(g)^d$.
\end{theo}
\begin{proof}
Let $x_0 \in \mathfrak{m}_{\overline{K}}$, $x_0 \neq 0$, such that $[\alpha](x_0) = 0$ and let $u_0 = S(x_0)$. Since $\Gamma_{\alpha} \circ R = R \circ [\alpha]$, we get $\Gamma_{\alpha}(u_0) = 0$. 

We will now prove that $u_0$ is a uniformizer of $E:= K(u_0)$. The discussion following \cite[Prop. 6.4]{laubie2002systemes} (it concerns the classical Lubin-Tate case but its generalization to relative Lubin-Tate groups is straightforward), for $d$ prime to $p$, shows that $u_0 = R(x_0)$ is a root of multiplicity $d$ of $\Gamma_{\alpha}$ and that the other roots of $\Gamma_\alpha$ other than $0$ are also of multiplicity $d$. In particular, we have that $[K(u_0):K] \leq \frac{q-1}{d}$. Since $R(x_0) = u_0$ and one can write $R(T)$ as $R(T)=(-1)^{d-1}T^d(1+R_1(T))$ with $R_1$ invertible in $\mathcal{O}_K[\![T]\!]$ (see the proof of \cite[Prop. 2.6]{laubie2002systemes}). This implies that $[K(x_0):K(u_0)] \leq d$. Since $x_0$ generates the first level of the relative Lubin-Tate extension $K_1$ over $K$, we have $[K(x_0):K] = q-1$. It follows that the previous inequalities on the degree of the extensions are equalities, and so $[K(x_0):K(u_0)] = d$. Moreover, since $K(x_0)/K$ is totally ramified, so is $K(x_0)/K(u_0)$. Since $R(T)=(-1)^{d-1}T^d(1+R_1(T))$ with $R_1$ invertible in $\mathcal{O}_K[\![T]\!]$, we get that $v_{K(x_0)}(u_0)=d$ and so $u_0$ is a uniformizer of $K(u_0)$.

Let us now define a sequence $(x_n)$ such that $[\alpha](x_{n+1})=x_n$, and let $u_n=R(x_n)$. Then 
$$\Gamma_{\alpha}(u_n) = \Gamma_{\alpha} \circ R(x_n) = R \circ [\alpha](x_n) = R(x_{n-1}) = u_{n-1}.$$

Since the power series $\Gamma_\alpha$ satisfies conditions of definition \ref{defphiit} and since $u_0$ is a uniformizer of $L$, the extension $L_\infty:=\bigcup L_n$ where $L_n=L(u_n)$ is $\phi$-iterate. Since $u_n=R(x_n)$ we have $E_\infty \subset K_\infty$ and since $K_\infty$ is an abelian extension of $K$, so is $E_\infty$. 

Moreover, $\Gamma_k(u_n) = \Gamma_k \circ R(x_n) = R \circ [k](x_n) = R(g(x_n))$ if $g \in \Gal(K_{\infty}/K)$ is such that $\chi_{\alpha}(g) = k$ and $R(g(x_n)) = g(R(x_n)) = g(u_n)$, and so $\Gamma_k(u_n) = g(u_n)$ for $\chi_\alpha(g) = k$. Now let $\eta$ be the character $\Gal(K_\infty/K) \to \mathcal{O}_F^\times$, given by $g \mapsto \Gamma_{\chi_{\alpha}(g)}'(0)=\chi_{\alpha}(g)^k$. Since $\Gal(K_{\infty}/K) \simeq \mathcal{O}_F^{\times}$ via $g \mapsto \chi_{\alpha}(g)$, this character is surjective onto $\left\{x^d~: x \in \mathcal{O}_K^{\times}\right\}$, which is a subgroup of $\mathcal{O}_F^\times$ with finite index $d$. If $g \in \Gal(K_\infty/K)$ is such that $\eta(g)=1$, then $\chi_\alpha(g)^d=1$ and so $\Gamma_{\chi_\alpha(g)}(T)=T$ since the power series $\Gamma_k$ are determined by their derivative at $0$ by lemma \ref{Gamma dérivée en 0}. In particular, $\Gamma_{\chi_{\alpha}(g)}(u_n) = u_n$ for all $n$ and so $g \in \Gal(K_{\infty}/E_{\infty})$. If $g \in \Gal(K_\infty/E_\infty)$, then for all $n$, $g(u_n) = u_n$ since $E_\infty$ is generated by the $u_n$ over $K$. But $g(u_n) = g(R(x_n))=R(g(x_n))=\Gamma_{\chi_\alpha(g)}(u_n)$ and so $\Gamma_{\chi_\alpha(g)}(T)-T$ is a power series with coefficients in $\mathcal{O}_K[\![T]\!]$ which vanishes on an infinite subset of the unit disk of $\Cp$, and so $\Gamma_{\chi_\alpha(g)}(T)=T$ and $\eta(g)=1$.
\end{proof}

In particular, we deduce the following result which shows that in a general setting there exist extensions $K_\infty/K$ for which the lift of $\Gamma_K$ is of finite height and for which there exists extensions $L_\infty/K_\infty$ that are not unramified but such that the action of $\Gamma_L$ is liftable of finite height:
\begin{coro}
There exists a finite extension $K$ of $\Qp$ and a $\phi$-iterate extension $K_\infty/K$ and a finite extension $E$ of $K$ such that, if $u$ denotes the image of $T$ by the embedding of proposition \ref{uembedding} for $\A_K$ and if $v$ denotes the image of $T$ by the embedding of proposition \ref{uembedding} for $\A_E$, then $\mathcal{O}_E[\![u]\!] \neq \mathcal{O}_E[\![v]\!]$.
\end{coro}
\begin{proof}
Let $K/\Qp$ be a finite extension. Let $E_\infty/K$ be a relative Lubin-Tate extension attached to a relative Lubin-Tate group, relative to an extension $K/F$ where $K/F$ is unramified, and let $W$ be a finite subgroup of $\mathcal{O}_K^\times$ with cardinal prime to $p$ (for example, one can take the set of $(p-1)$-nth roots of $1$). Let $K_\infty/K$ be the invariants under $W$ of $E_\infty$. Then there exists a finite extension $K'$ of $K$, contained in $K_\infty$ such that $K_\infty/K'$ is $\phi$-iterate. Moreover, $E_\infty/E_1$ is $\phi$-iterate. Let $v$ be the element of $\Atplus$ which is the image of $T$ by the embedding given by proposition \ref{uembedding} for the extension $E_\infty$ and let $u$ be the element of $\Atplus$ which is the image of $T$ by the embedding given by proposition \ref{uembedding} for the extension $K_\infty$. Then the rings $\mathcal{O}_E[\![u]\!]$ and $\mathcal{O}_E[\![v]\!]$ can not be equal since $u$ is invariant under $W$ but $v$ is not.
\end{proof}

\section{Formal groups, semi-conjugation and condensation}
In this section, we show that the power series $F_g$ and $P$ given in the case of a finite height lift of the field of norms are related to endomorphisms of a formal group. 

Let us first recall the setting of non archimedean dynamical systems as studied in \cite{lubin1994nonarchimedean}. We say that a power series $g(T) \in T\cdot\mathcal{O}_K[\![T]\!]$ is invertible if $g'(0) \in \mathcal{O}_K^\times$ and noninvertible if $g'(0) \in \mathfrak{m}_K$. We say that $g$ is stable if $g'(0)$ is neither $0$ nor a root of $1$. The case studied by Lubin is the case where two stable power series $f$ and $u$ commute, with $f$ noninvertible and $u$ invertible. Following Lubin's statement that ``experimental evidence seems to suggest that for an invertible series to commute with a noninvertible series, there must be a formal group somehow in the background'', it seems reasonable to expect that, given such power series $f$ and $u$ (at least when $f$ is non zero modulo $\mathfrak{m}_K$), there exists a formal group $S$, a nonzero power series $h$ and two endomorphisms $f_S$ and $u_S$ of $S$ such that $f \circ h = h \circ f_S$ and $u \circ h = h \circ u_S$. Following the definitions of \cite{Li97}, we say in that case that $f$ and $f_S$, and $u$ and $u_S$, are semi-conjugate and that $h$ is an isogeny from $f_S$ to $f$.

One of the main examples of this situation is the condensation in the sense of Lubin (see \cite[p. 
144]{lubin1994nonarchimedean}), of which the construction of the power series $\Gamma_a$ in the previous section is a particular case. 

We will see how the power series $F_g$ given in the case of a finite height lift of the field of norms can be seen as semi-conjugates of endomorphisms of a relative Lubin-Tate group. But first, we will prove our theorem \ref{theo LT ext unramified}. Let $K_\infty/K$ be a relative Lubin-Tate extension, relative to a subfield $F$ of $K$ and attached to $\alpha \in \mathcal{O}_F$. Let $u_0 = 0$ and let $(u_n)_{n \in \N}$ be a compatible sequence of roots of iterates of $[\alpha]$, that is such that $[\alpha](u_{n+1})=u_n$ and $u_1 \neq 0$. Let $\overline{u} = (u_0,\cdots) \in \Etplus$. By §9.2 of \cite{Col02}, there exists $u \in \Atplus$, whose image in $\Etplus$ is $\overline{u}$, and such that $\phi_q(u) = [\alpha](u)$ and $g(u) = [\chi_\alpha(g)](u)$ for $g \in \G_K$. Let $\A_K$ be the $p$-adic completion of $\mathcal{O}_K[\![u]\!][1/u]$ inside $\At$. This element $u$ is exactly the image of $T$ by the embedding of proposition \ref{uembedding}, where $P(T)=[\alpha](T)$ and $F_g(T) = [\chi_\alpha(g)](T)$. Note that we have $\A_K^+=\mathcal{O}_K[\![u]\!]$. Now let $E$ be a finite extension of $K$, and let $E_\infty= E \cdot K_\infty$. By lemma \ref{lemm existunique extension étale cohen}, there exists a unique étale extension of $\varpi_K$-Cohen rings $\A_E$ of $\A_K$ lifting the extension $\E_E/\E_K$. Let $v \in \Atplus$ be such that $\A_E^+ = \mathcal{O}_K[\![v]\!]$. Since $u$ belongs to $\A_K^+ \subset \A_E^+$, we can write $u = h(v)$ for some $h \in \mathcal{O}_K[\![T]\!]$. In particular, one has $\theta \circ \phi_q^{-n}(u) = h(\theta \circ \phi_q^{-n}(v))$, so that the elements $v_n = \theta \circ \phi_q^{-n}(v)$ belong to $\Qpbar$. Let $E_n = E(v_n)$, and let $w \in \mathcal{O}_E \otimes W(\Etplus)$ be such that $\mathcal{O}_E[\![w]\!] = \{ x \in \mathcal{O}_E \otimes W(\Etplus) , \theta \circ \phi_q^{-n}(x) \in \mathcal{O}_{E_n} \textrm{ for all } n \geq 1 \}$. Such a $w$ exists since $v$ belongs to that set, and we have $\mathcal{O}_E[\![w]\!] \supset \mathcal{O}_E[\![v]\!] \supset \mathcal{O}_E[\![u]\!]$. In particular, we can write $u=f(w)$ with $f \in \mathcal{O}_E[\![T]\!]$. By lemma \ref{ajuster w} and lemma \ref{lemm H_g(w)}, one can assume that $g(w) = H_g(w)$ with $H_g(T) \in T\cdot \mathcal{O}_E[\![T]\!]$ and that $\phi_q(w) = Q(w)$ with $Q \in T\cdot \mathcal{O}_E[\![T]\!]$. Moreover, by lemma \ref{kappa injectif}, $\kappa(g)=H_g'(0)$ defines an injective character: $\Gamma_E \rightarrow \mathcal{O}_E^\times$, where $\Gamma_E=\Gal(E_\infty/E)$.

\begin{lemm}
One has $f(0)=0, f'(0)\neq 0$ and $\kappa = \chi_\alpha$. 
\end{lemm}
\begin{proof}
We have $g(u) = F_g(u) = F_g \circ f(w)$ and $g(u) = g(f(w)) = f(g(w))=f \circ H_g(w)$, so that $F_g \circ f = f \circ H_g$. We know that we can write $F_g(T)=\chi_\alpha\cdot T+O(T^2)$ and $H_g(T) = \kappa(g)\cdot T+O(T^2)$. Write $f(T) = \alpha_kT^k+O(T^{k+1})$, with $\alpha_k \neq 0$. The equality $F_g \circ f = f \circ H_g$ gives us that $\alpha_k\chi_\alpha(g)T^k=\alpha_k\kappa(g)T^k$, so that $\chi_\alpha(g)=\kappa(g)^k$, and since $\chi_\alpha : \Gamma_E \to \mathcal{O}_F^\times$ is injective, this implies that $f(0)=0$. 

It remains to prove that $k=1$ in order to prove the lemma. By proposition \ref{expression caractere eta}, if $\lambda$ is a uniformizer of $\mathcal{O}_E$ such that $E_\infty \subset E_\lambda$, then there exists a subfield $F'$ of $E$ and an integer $r \geq 1$ such that $\kappa = N_{E/F'}(\chi_\lambda^E)^r$. Since $\chi_\alpha = \kappa^k$, this implies that $F \subset F'$. However, $\Gamma_E$ is isomorphic to an open compact subgroup of $\mathcal{O}_F^\times$ by $\chi_\alpha$ (and the embedding $\Gamma_E \hookrightarrow \Gamma_K$) and to an open compact subgroup of $\mathcal{O}_{F'}^\times$ by $\kappa$, which are analytic groups of respective dimension $[F:\Qp]$ and $[F':\Qp]$. Since $F \subset F'$, this implies that $F=F'$. Since we have $\chi_\alpha = N_{E/F}(\chi_\lambda^E)$, this implies that $\kappa = \chi_\alpha$ and $k=r=1$, so that $f'(0) \neq 0$.
\end{proof}

Following Li's terminology, $f$ is an isogeny of order $1$ from $Q$ to $[\alpha]$, and we can apply his results of \cite[§3]{Li97}.

\begin{prop}
\label{prop u=w mod Frob}
There exists $n \geq 0$ such that $\mathcal{O}_E[\![w]\!] = \phi_q^{-n}(\mathcal{O}_E[\![u]\!])$.
\end{prop}
\begin{proof}
By \cite[Thm. 3]{Li97}, there exists $\check{f} \in \mathcal{O}_{E}[\![T]\!]$ and $n \in \N$ such that $\check{f} \circ f = Q^{\circ n}$. Evaluating at $w$ gives us $\check{f}(u) = \phi_q^{-n}(w)$, so that $w \in \phi_q^{-n}(\mathcal{O}_E[\![u]\!])$.
\end{proof}

\begin{coro}
We have $\A_K^+ = \A_E^+$. 
\end{coro}
\begin{proof}
One has $\A_K^+ \subset \A_E^+ \subset \phi_q^{-n}(\A_K^+)$ by the previous proposition. In particular, one can write $v=h(\phi_q^{-n}(u))$ with $h \in \mathcal{O}_E[\![T]\!]$. This implies that $\overline{v}$ belongs to a purely inseparable extension of $k_E[\![\overline{u}]\!]$. However, by definition of $v$ and the theory of field of norms, $\overline{v}$ belongs to a separable extension of $k_E[\![\overline{u}]\!]$, so that $\overline{v} \in k_E[\![\overline{u}]\!]$. This implies that $\A_K^+/\varpi_E\A_K^+ = \A_E^+/\varpi_E\A_E^+$ and therefore that $\A_K^+ = \A_E^+$.
\end{proof}

In particular, this answers our question \ref{question unramified} in the case of Lubin-Tate extensions: the only extensions of $K_\infty$ for which there exists a finite height lift are the unramified ones. 

We will now show how to relate the power series given in the case of a finite height lift of the field of norms with endomorphisms of some relative Lubin-Tate group. To do so, we will use the ``canonical Cohen ring'' for norm fields constructed by Cais and Davis in \cite{cais2015canonical}. Let $K_\infty/K$ be an infinite strictly APF extension which is Galois, and such that the action of $\Gamma_K = \Gal(K_\infty/K)$ is liftable on $\A_K = \widehat{\mathcal{O}_E[\![T]\!][1/T]}$ with power series $P(T)$ and $\{F_g(T)\}_{g \in \Gamma_K}$. As in section \ref{relevementHT}, we can assume that $E$ contains $K$, is Galois over $\Qp$ and is such that $E_\infty/E$ is totally ramified. By proposition \ref{existvtau}, there exists $u_\tau \in \Atplus$ such that $\phi_q(u_\tau) = P^\tau(u_\tau)$, $g(u_\tau) = F_g^\tau(u_\tau)$ and the elements $v_n:=\theta \circ \phi_q^{-n}(u_\tau)$ belong to $\Qpbar$. Let $\kappa$ be the character $\Gamma_E \to \mathcal{O}_E^\times$ given by $\kappa(g)=(F_g^\tau)'(0)$. By theorem \ref{thm phigamma implique LT}, there exists a subfield $F$ of $E$ and a relative Lubin-Tate group $S$, relative to the extension $F^{\mathrm{unr}} \cap E$ and attached to an element $\alpha$ of $F$, such that if $E_\infty^S$ denotes the extension of $E$ generated by the torsion points of $S$, then $E_\infty \subset E_\infty^S$ and the extension $E_\infty^S/E_\infty$ is finite. Moreover, there exists $r \geq 1$ such that $\kappa = \chi_\alpha^r$. Up to replacing $E$ by a finite extension of $E$, we can also assume that $E_\infty^S=E_\infty$ and we do so in what follows. Let $w_0 = 0$ and let $(w_n)_{n \in \N}$ be a compatible sequence of roots of iterates of $[\alpha]$, that is such that $[\alpha](w_{n+1})=w_n$ and $w_1 \neq 0$. Let $\overline{w} = (w_0,\cdots) \in \Etplus$. By §9.2 of \cite{Col02}, there exists $w \in \Atplus$, whose image in $\Etplus$ is $\overline{w}$, and such that $\phi_q(w) = [\alpha](w)$ and $g(w) = [\chi_\alpha(g)](w)$ for $g \in \G_K$. We now prove a link between the ring that we called $\A_E^+$ and the canonical Cohen ring $\A_{E_\infty/E}^+$ of Cais and Davis:

\begin{lemm}
\label{lemma canonical ring = ring LT upto Frob}
There exists $k \geq 0$ such that $\A_{E_\infty/E}^+ = \phi_q^{-k}(\mathcal{O}_E[\![w]\!])$.
\end{lemm}
\begin{proof}
Let $E'= F^{\mathrm{unr}} \cap E$ and let $F_\infty^S$ be the extension of $F$ generated over $E'$ by the torsion points of $S$. For $n\geq 1$, let $F_n:=E'(w_n)$, so that $F_\infty^S=\bigcup_{n \geq 1}F_n$. Let $F_0 = E'$.

The computation of the ramifications subgroups of the Galois group of a relative Lubin-Tate extension (see for example \cite[Prop. 3.7]{SchneiderLTcours}) shows that the sequence $\{F_n\}_{n \geq 0}$ is the tower of elementary extensions attached to $F_\infty^S/E'$ defined in \cite[1.4]{Win83}. Following definition \ref{defi canonical CaisDavis}, this implies that the element $w \in \Atplus$ belongs to the ring $\A_{F_\infty^S/E'}^+$. By definition \ref{defi canonical CaisDavis}, the image of $\A_{F_\infty^S/E'}^+$ modulo a uniformizer is contained inside $\E_{F_\infty^S/E'}^+$. Moreover, since $\overline{w}$ is a uniformizer of the ring $\E_{F_\infty^S/E'}^+$ and since $\A_{F_\infty^S/E'}^+$ is an $\mathcal{O}_{E'}$-algebra which is separated and complete for the $p$-adic topology, this implies that $\A_{F_\infty^S/E'}^+=\mathcal{O}_{E'}[\![w]\!]$.

Let $w'$ be such that $\A_{E_\infty/E}^+ = \mathcal{O}_E[\![w']\!]$. By proposition \ref{prop embedding canonical Cohen ring}, we have $w \in \mathcal{O}_E[\![w']\!]$, and the same proof as the one of proposition \ref{prop u=w mod Frob} shows that $\mathcal{O}_E[\![w']\!] = \phi_q^{-n}(\mathcal{O}_E[\![w]\!])$. 
\end{proof}

The following proposition allows us to see some power of the Frobenius of $u_\tau$ inside the ring $\A_{E_\infty/E}^+$:
\begin{prop}
\label{prop v_tau inside Cohen ring}
There exists $i \geq 0$ such that $\phi_q^i(u_\tau) \in \A_{E_\infty/E}^+$.
\end{prop}
The strategy to prove this proposition is basically the same as the one proposed by Cais and Davis in \cite[§7]{cais2015canonical} to show that there was a lift of a uniformizer to $\A_{L/K}^+$ for $\phi$-iterate extensions. Since most of the arguments are the same, we will only explain the points where we have to adapt the proof. The idea is to relate in some way the fields $E(v_n)$ and the tower $\{E^{(m)}\}$ of elementary extensions extensions attached to the APF extension $E_\infty/E$.

\begin{lemm}
\label{lemma n_0 tot ram of deg q}
There exists $n_0 \geq 0$ such that for all $n \geq n_0$, the extension $E(v_{n+1})/E(v_n)$ is totally ramified of degree $q$ and such that $v_{E(v_n)}(v_n)$ is independent of $n \geq n_0$.
\end{lemm} 
\begin{proof}
We have $P^\tau(v_{n+1})=v_n$, with $P^\tau \in T\cdot \mathcal{O}_E[\![T]\!]$ of Weierstrass degree $q$. Since $v_n = \overline{u}_n \mod p^c$, we know that $v_E(v_n) \rightarrow 0$ as $n \rightarrow +\infty$. The fact that for $n$ big enough, $E(v_{n+1})/E(v_n)$ is totally ramified of degree $q$ then follows from the theory of Newton polygons. 
Now let $n_0 \geq 0$ be such that, for all $n \geq n_0$, $E(v_{n+1})/E(v_n)$ is totally ramified of degree $q$ and such that $v_p(v_n) < c$. We then have $v_p(v_{n+1}) = v_p(\overline{u}_{n+1}) = \frac{1}{q}v_p(\overline{u}_n) = \frac{1}{q}v_p(v_n)$, so that, since $E(v_{n+1})/E(v_n)$ is totally ramified of degree $q$, $v_{E(v_n)}(v_n)$ is independent of $n \geq n_0$, and is equal to $v_{E_{v_{n_0}}}(v_{n_0})$.
\end{proof}

Let $\pi_n$ be a uniformizer of $\mathcal{O}_{E(v_n)}$ such that $E(v_n) = E(\pi_n)$ (this is possible since $E(v_n)/E$ is totally ramified).
\begin{lemm}
\label{lemm positive constants ord pi_n sigma}
There exist positive constants $A$ and $B$ such that, for all $n \geq 1$, for all $\sigma \in \Gal(E(v_n)/E(v_{n-1})$, 
$$Aq^n \leq v_{E(v_n)}(\sigma(\pi_n)-\pi_n) \leq Bq^n.$$
\end{lemm}
\begin{proof}
Using the same arguments as in the proof of \cite[Lemma 7.2]{cais2015canonical}, we can prove that there exist positive constants $A'$ and $B'$ such that, for all $n \geq 1$ and for all $\sigma \in \Gal(E(v_n)/E(v_{n-1}))$:
$$A'q^n \leq \mathrm{ord}_{v_n}(\sigma(v_n)-v_n) \leq B'q^n.$$
Let $n_0$ be such as in lemma \ref{lemma n_0 tot ram of deg q}. Since $v_{E(v_n)}(v_n) = k_0$ and since $\pi_n$ is a uniformizer of $\mathcal{O}_{E(v_n)}$, one can write
$$v_n = \sum_{k \geq k_0}a_k\pi_n^k,$$
where the $a_k$ can be written as Teichmüller representatives of elements of $k_E$ (in particular, $\sigma(a_k)=a_k$ for all $k$ and $v_E(a_k)=0$ if $a_k \neq 0$), and with $a_{k_0} \neq 0$.
This gives us that
$$\sigma(v_n)-v_n = \sum_{k \geq k_0}a_k(\sigma(v_n)^k-v_n^k)$$
so that
$$\sigma(v_n)-v_n = (\sigma(\pi_n)-\pi_n)\sum_{k \geq k_0}a_k(\sum_{m=0}^k\sigma(\pi_n)^m\pi_n^{k-m}).$$
In particular, we get that
$$\mathrm{ord}_{\pi_n}(\sigma(v_n)-v_n) = \mathrm{ord}_{\pi_n}(\sigma(\pi_n)-\pi_n)+k_0,$$
so that $\mathrm{ord}_{\pi_n}(\sigma(\pi_n)-\pi_n)+k_0 = k_0\mathrm{ord}_{v_n}(\sigma(v_n)-v_n)$. In particular, by letting $B=k_0(B'+1)$, we get that $\mathrm{ord}_{\pi_n}(\sigma(\pi_n)-\pi_n) \leq Bq^n$. There also exists $n_1 \geq 0$ such that $q^{n_1} > A'$, so that $k_0(A'-\frac{1}{q^n}) > 0$ for $n \geq n_1$, and if we let 
$$A:=\min( \frac{1}{q^n}\min_{n \leq n_1, \sigma \in \Gal(E(v_n)/E(v_{n-1})}(\mathrm{ord}_{\pi_n}(\sigma(\pi_n)-\pi_n)),k_0(A'-\frac{1}{q^{n_1}})),$$
then we have $q^nA \leq \mathrm{ord}_{\pi_n}(\sigma(\pi_n)-\pi_n)$ for all $n \geq 1$ and $\sigma \in \Gal(E(v_n)/E(v_{n-1}))$.
\end{proof}

\begin{lemm}
\label{lemma true for all sigma in Gal(E(vn)/E)}
Let $\sigma \in \Gal(E(v_n)/E)$. Then
$$v_{E(v_n)}(\sigma(\pi_n)-\pi_n) \leq Bq^n,$$
where $B$ is the same as in lemma \ref{lemm positive constants ord pi_n sigma}.
\end{lemm}
\begin{proof}
The same proof as the one of \cite[Lemm. 7.3]{cais2015canonical} shows that for all $\sigma \in \Gal(E(v_n)/E)$,
$$\mathrm{ord}_{v_n}(\sigma(v_n)-v_n) \leq B'q^n,$$
with $B'$ the same constant as the one in the proof of lemma \ref{lemm positive constants ord pi_n sigma}. Using the fact that $\mathrm{ord}_{\pi_n}(\sigma(\pi_n)-\pi_n)+k_0 = k_0\mathrm{ord}_{v_n}(\sigma(v_n)-v_n)$, we get that $\mathrm{ord}_{\pi_n}(\sigma(\pi_n)-\pi_n) \leq Bq^n$ with $B=k_0(B'+1)$, which is the same constant $B$ as in lemma \ref{lemm positive constants ord pi_n sigma}.
\end{proof}

\begin{proof}[Proof of Proposition \ref{prop v_tau inside Cohen ring}]
Lemma \ref{lemm positive constants ord pi_n sigma} and lemma \ref{lemma true for all sigma in Gal(E(vn)/E)} allow us to apply lemmas 7.5, 7.6, corollary 7.7, proposition 7.9 and lemmas 7.10 and 7.11 of \cite{cais2015canonical} (note that we do not need lemma 7.4 of ibid. since all our extensions are already Galois) in order to obtain the equivalent of corollary 7.12 of ibid., that is that there exist $i, i_0 \geq 0$ such that if $q^m|[E^{(i+j)}:E^{(i)}]$, then $\pi_{i_0+m}$ belongs to $E^{(i+j)}$ and \textit{a fortiori} $v_{i_0+m} \in E^{(i+j)}$. In particular, if $i_1$ is such that $[E^{(i)}:E]|q^{i_1}$, then $\phi_q^{i_1}(u_\tau) \in \A_{E_\infty/E}^+$.
\end{proof}

Lemma \ref{lemma canonical ring = ring LT upto Frob} and proposition \ref{prop v_tau inside Cohen ring} now allow us to relate the power series $F_g^\tau$ and $P^\tau$ to the power series $[\alpha](T)$ and $[\chi_\alpha(g)](T)$:

\begin{theo}
If $K_\infty/K$ is an infinite strictly APF Galois extension such that the action of $\Gamma_K = \Gal(K_\infty/K)$ is liftable of finite height with power series $P(T)$ and $\{F_g(T)\}_{g \in \Gamma_K}$, then the power series $F_g$ are twists of semi-conjugates over a finite extension of $\Qp$ of endomorphisms of a relative Lubin-Tate group $S$, relative to a subfield $F$ of a Galois extension $E$ of $\Qp$ and attached to an element $\alpha$ of $F$, and such that if $E_\infty^S$ denotes the extension of $E$ generated by the torsion points of $S$, then $K_\infty \subset E_\infty^S$ and $E_\infty/K_\infty$ is finite. Moreover, there is an isogeny from $[\alpha]$ to $P^\tau$.  
\end{theo}
\begin{proof}
Lemma \ref{lemma canonical ring = ring LT upto Frob} and proposition \ref{prop v_tau inside Cohen ring} show that there exists $i \geq 0$ and $h(T) \in \mathcal{O}_E[\![T]\!]$ such that $u_\tau = h(\phi_q^{-i}(w))$. Let $\tilde{w}=\phi_q^{-i}(w)$, so that $u_\tau = h(\tilde{w})$. For $g \in \G_E$, we have $g(u_\tau) = F_g^\tau(u_\tau)$, and $\phi_q(u_\tau) = P^\tau(u_\tau)$. We also have $g(w) = [\chi_\alpha(g)](w)$ and $\phi_q(w)=[\alpha](w)$ so that $g(\tilde{w}) = [\chi_\alpha(g)](\tilde{w})$ and $\phi_q(\tilde{w})=[\alpha](\tilde{w})$. This gives us that $F_g^\tau \circ h = h \circ [\chi_\alpha(g)]$ and that $P^\tau \circ h = h \circ [\alpha]$, which is what we wanted to prove.
\end{proof}

\begin{rema}
The fact that the power series $F_g$ are twists of semi-conjugates of the relative Lubin-Tate group seems necessary, as there is no reason to expect that the character $g \mapsto F_g'(0)$ has weight $\geq 1$ at $\id$. For example, let $K_\infty/K$ be a (classical) Lubin-Tate extension, attached to a uniformizer $\pi$ of $\mathcal{O}_K$, where $K/\Qp$ is totally ramified and Galois, and let $u \in \At = \mathcal{O}_K \otimes_{\Qp}W(\Etplus)$ be the element constructed in the Lubin-Tate setting at the beginning of this section, so that $[\chi_\pi(g)](u) = g(u)$ for $g \in \G_K$. For any $\tau \neq \id$ in $\Gal(K/\Qp)$, the element $u_\tau = (\tau \otimes \id)(u)$ satisfies $F_g^\tau(u_\tau) = g(u_\tau)$, $P^\tau(u_\tau) = \phi(u_\tau)$ and $\overline{u_\tau} = \overline{u}$, so that the action of $\Gamma_K$ is liftable with power series $[\pi]^\tau$ and $\{[\chi_\pi(g)]^\tau\}_{g \in \Gamma_K}$, yet the weight of the character $\eta_\tau : g \mapsto ([\chi_\pi(g)]^\tau)'(0)$ at $\id$, which is the weight of $\eta_\id : g \mapsto ([\chi_\pi(g)])'(0)$ at $\tau$, is zero (see for example \cite[§4]{Tat67} or §3.2 of \cite{fourquaux2009applications}).

However, in the case where $K/\Qp$ is unramified, the author expects that the power series $F_g$ are actual semi-conjugates of the relative Lubin-Tate group.
\end{rema}

\bibliographystyle{amsalpha}
\bibliography{bibli}
\end{document}